\documentclass[11pt, letterpaper]{amsart}
\usepackage{standard}
\usepackage{mathtools}
\usepackage{hyperref}
\usepackage{cleveref}

\usepackage{xcolor}

\DeclareMathOperator{\proj}{proj}

\renewcommand{\emptyset}{\varnothing}

\newcommand{\la}{\langle}
\newcommand{\ra}{\rangle}

\newcommand{\crap}{Z} 
\newcommand{\taub}{\tau}
\newcommand{\xib}{\xi}
\newcommand{\etab}{\eta}
\newcommand{\taubhat}{\hat{\tau}}
\newcommand{\xibhat}{\hat{\xi}}
\newcommand{\etabhat}{\hat{\etab}}
\DeclareMathOperator{\sgn}{sgn}
\newcommand{\hamvf}{{}^bH}
\newcommand{\taue}{\tau_e}
\newcommand{\xie}{\xi_e}
\newcommand{\etae}{\eta_e}
\newcommand{\bu}{{bu}}

\DeclareMathOperator{\liptic}{ell}
\newcommand{\Hone}{\mathcal{H}}
\newcommand{\Hmone}{\mathcal{H^*}}
\newcommand{\strng}{\mathsf{S}}
\newcommand{\Tbstar}{{}^bT^*}
\newcommand{\Tb}{{}^bT}
\newcommand{\Testar}{{}^eT^*}
\newcommand{\Tbdotstar}{{}^b\dot{T}^*}
\newcommand{\Sbdotstar}{{}^b\dot{S}^*}

\newcommand{\Te}{{}^eT}

\newcommand{\fib}{\mathsf{F}}
\newcommand{\fcal}{\mathscr{F}}
\newcommand{\Psib}{\Psi_{\bu}}
\newcommand{\Phib}{{}^b\Phi}
\newcommand{\sigmab}[1]{\sigma_b^{#1}}
\DeclareMathOperator{\Opb}{Op_b}
\DeclareMathOperator{\tOpb}{\widetilde{Op}_b}
\DeclareMathOperator{\Diff}{Diff}

\newcommand{\V}{\mathcal{V}}

\newcommand{\Hb}{H_{b\F}}
\newcommand{\HbHone}{H_{b\F,\Hone}}
\newcommand{\HbHonestar}{H_{b\F,\Hone^*}}

\newcommand{\Omegabar}{\overline{\Omega}}

\newcommand{\hilbX}{\mathcal{X}}
\newcommand{\hilbY}{\mathcal{Y}}

\newcommand{\opWFb}{\operatorname{WF}'_{b}}
\newcommand{\WFb}{\operatorname{WF}_{b\F}}
\newcommand{\WFbh}{\operatorname{WF}_{b\F,\Hone}}

\newcommand{\WFbhstar}{\operatorname{WF}_{b\F,\Hmone}}
\newcommand{\WFbhu}{\operatorname{WF}_{b\F,\Hone}}
\newcommand{\WFbhstaru}{\operatorname{WF}_{b\F,\Hmone}}
\newcommand{\Sestar}{{}^eS^*}
\newcommand{\Sbstar}{{}^bS^*}

\newcommand{\unif}{u}

\newcommand{\Lambdaplus}{\Lambda_{ \ell} }
\newcommand{\Lambdaminus}{\Lambda_{-\ell} }
\newcommand{\Lambdahalf}{\Lambda_{\frac{1}{2}}}
\newcommand{\Lambdaneghalf}{\Lambda_{-\frac{1}{2}}}

\newcommand{\ind}{\varpi}
\newcommand{\C}{\mathsf{A}}

\newcommand{\IC}{\mathcal{I}}

\newcommand{\E}{\mathscr{E}}

\newcommand{\M}{\mathscr{M}}

\DeclareMathOperator{\Id}{Id}
\newcommand{\tP}{P_W}
\newcommand{\tQ}{\widetilde{Q}}

\newcommand{\F}{\mathsf{F}}
\newcommand{\Psibf}{\Psi_{b\F}}
\newcommand{\Diffbf}{\operatorname{Diff}_{b\F}}
\newcommand{\CIf}{\mathcal{C}^\infty_{\fib}}
\newcommand{\WFbf}{\operatorname{WF}_{b\F}}

\newcommand{\pert}{\Upsilon}
\newcommand{\hypsurf}{Y}

\DeclareMathOperator{\spn}{span}

\newtheorem{dullhyp}{Hypothesis}
\newenvironment{hyp}[2][]
  {\renewcommand\thedullhyp{#2}\begin{dullhyp}[#1]}
  {\end{dullhyp}}

\numberwithin{theorem}{section}

\title{Wave propagation on rotating cosmic string spacetimes}
\author{Jared Wunsch}
\address{Department of Mathematics\\Northwestern University\\Evanston
  IL 60208}
\email{jwunsch@math.northwestern.edu}
\author{Katrina Morgan}
\address{Department of Mathematics\\Temple University\\Philadelphia PA 19122}
\email{morgank@temple.edu}
\thanks{The authors thank Jeff Galkowski, Andr\'as Vasy, and Maciej Zworski for helpful
  conversations.  An anonymous referee made a number of corrections
  and useful suggestions.   KM was partly supported by NSF Postdoctoral
  Fellowship DMS--2002132.   JW was partially supported
by Simons Foundation grant 631302, NSF grant DMS--2054424, and a
Simons Fellowship. }

\begin{document}

\begin{abstract}
	A rotating cosmic string spacetime has a singularity along a timelike curve corresponding to a one-dimensional source of angular momentum. Such spacetimes are not globally hyperbolic: they admit closed timelike curves near the string. This presents challenges to studying the existence of solutions to the wave equation via conventional energy methods. In this work, we show that semi-global forward solutions to the wave equation do nonetheless exist, but only in a microlocal sense.  The main ingredient in this existence theorem is a propagation of singularities theorem that relates energy entering the string to energy leaving the string.  The propagation theorem is localized in the fibers of a certain fibration of the blown-up string, but global in time, which means that energy entering the string at one time may emerge previously.
\end{abstract}

\maketitle
\tableofcontents

\section{Introduction}
\subsection{Rotating cosmic string metrics and wave propagation}
Cosmic string spacetimes are cosmological models that feature
singularities along timelike curves (``strings''\footnote{Not to be
  confused with the superstrings of high energy particle physics.}).  Starting with work of Kibble
\cite{Ki:76}, physicists have speculated on their formation in the
early universe, and detection of these structures, or bounds on their
prevalence, remain subjects of active current experimental research
\cite{LI:21}.  In work of Deser--Jackiw--'t Hooft \cite{DeJaTh:84},
the simplest cosmic string solutions are viewed as solutions to the
Einstein equations in $2+1$ dimensions, with a third spatial
dimension along the string quotiented out.  Such solutions are, of
necessity, flat away from the singularity.  The solution corresponding
to a static string is then simply the product spacetime given by $\RR$
(time variable) times a flat 2d cone.  The simplest \emph{rotating string}
solution, however, is of a more interesting Lorentzian character,
given in cylindrical coordinates by the metric
\begin{equation}\label{metric}
g=(dr^2 + r^2 d\varphi^2)-(dt^2-2\C \, dt \, d\varphi+ \C^2d\varphi^2).
\end{equation}
Here $\C=-4GJ$ where $G$ is the gravitational constant and $J$ is the
angular momentum of the string.  This metric has two features of
unusual interest from the point of view of wave propagation: it is singular at $r=0,$ and it admits closed
timelike curves, hence is not globally hyperbolic.

The corresponding wave operator $\Box_g$ is given by
	\begin{equation} \label{operator}
	\begin{split}
		\Box_g &= -\left(1-\frac{\C ^2}{r^2}\right) \partial_t^2 + \Delta + \frac{2\C }{r^2} \partial_t\partial_\varphi \\
			&= -\partial_t^2 + r^{-2} (r\pa_r)^2 + r^{-2}(\C  \pa_t + \pa_\varphi )^2
	\end{split}
      \end{equation}
Owing to the absence of global hyperbolicity, we are unable to prove
existence of solutions to the wave equation by conventional energy
methods.  In this paper, we thus resort to proving \emph{microlocal}
energy estimates in order to deal with the propagation along rays
passing through the string at $r=0.$  We then use these estimates to
prove the existence of \emph{microlocally forward} solutions to the
inhomogeneous equation $\Box_g u=f$ (and perturbations thereof).
Such solutions have singularities only along the
  (asymptotically) forward flowout
  of the singularities of $f$, together with the forward flowout of
  the string itself.
      
\subsection{Propagation of singularities}
The singularities of the metric and wave operator give
an associated fibration of $\{r=0\}$ (best regarded as the front face
of a real blowup of the string): the helical fibers are integral curves of $\C
\pa_t + \pa_\varphi,$ or, equivalently, are level sets of 
$\varphi-t/\C\bmod 2 \pi \ZZ.$
        
In this paper we study lower-order perturbations of $\Box_g,$
with a
class including real-valued potentials and Klein--Gordon mass
parameters as well as certain magnetic potentials: let $\pert$ be a first
order 
differential operator of the form
\begin{equation}\label{pert}
\pert= f_1 \pa_t +f_2 \pa_\varphi+f_3 r \pa_r +f_4
\end{equation}
where $f_\bullet=f_\bullet(r,t,\varphi)$ are complex-valued smooth functions, with uniform
derivative bounds in $\RR\times [0, \infty) \times S^1$.
Let
  $$
P=\Box_g + \pert.
$$

Standard propagation of singularities results hold along all geodesics
not hitting the $r=0$ singularity of the metric.  At $r=0,$ however,
new techniques are required to obtain such propagation results, or, equivalently, microlocalized energy estimates.  Thus our first
result concerns the propagation of regularity through $r=0.$ Since,
away from $r=0,$ wavefront set propagates along null bicharacteristics
(a.k.a.\ lightlike geodesics), we are concerned with those null
bicharacteristics that reach $r=0.$ Along such bicharacteristics, the
$t$ variable is in fact monotone, with $dr/dt=\pm 1,$ $d\varphi/dt=0$
(see Section \ref{sectionBichars} for details).  We call such curves
``incoming'' or ``outgoing'' according to whether $dr/dt=\mp 1$ (the
choice of $-$ corresponds to incoming).  Different incoming and outgoing curves hit $r=0$ at
different angles $\varphi,$ different times $t,$ and with different
energies $\tau$ (dual variable to $t$ in the cotangent bundle).  We
let
$$
\fcal_{I/O, \varphi_0, \tau_0}
$$
denote the union of these incoming resp.\ outgoing null bicharacteristics, into or out
of a specified fiber with
$\varphi_0\equiv \varphi-t/\C,$ and with $\tau=\tau_0.$  We denote the
union of all incoming/outgoing null bicharacteristics by
$$
\fcal_{I/O}\equiv \bigcup_{ \varphi_0, \tau_0} \fcal_{I/O, \varphi_0, \tau_0}.
$$

We define a Sobolev space closely associated with the Dirichlet form
of $\Box_g:$
let the norm on $\Hone$ be defined as
	\[
		 \|u\|_{\Hone}^2 =  \|\pa_r u\|_{L^2(X)}^2 +
                 \|r^{-1}(\C \pa_t + \pa_\varphi)  u\|_{L^2(X)}^2 + \|\pa_tu\|_{L^2(X)}^2 + \| \pa_\varphi u\|_{L^2(X)}^2
          + \| u\|_{L^2(X)}^2.
         \]

In its simplest form, our propagation of singularities theorem is as follows:
\begin{theorem}\label{theorem:firstpropagation}
Let $u \in \Hone$ near $r=0$ and let $Pu=0.$ If the wavefront set of $u$ is disjoint from
$\fcal_{I, \varphi_0, \tau_0}$
uniformly in $t$ then $u$ has no wavefront set at
$\fcal_{O, \varphi_0, \tau_0},$ uniformly in $t.$
\end{theorem}
The notion of uniformity used here will be elucidated below; moreover
the full statement of the theorem, which is Theorem~\ref{theorem:reflection} below, involves a notion of wavefront set
that is appropriately defined down to $r=0$ (b-wavefront set relative
to $\Hone$).  The full statement of the theorem
also includes the inhomogeneous equation $Pu=f$ and deals with a
family of (b-)Sobolev-based wavefront sets, with orders of regularity measured relative to $\Hone$.  We can also relax the hypotheses on $u$ to allow for a
range of regularities, albeit always measured relative to $\Hone.$

The theorem can be regarded as an energy estimate: it says that
estimates along incoming rays down to $r=0$ can be propagated outward
from $r=0$ in a manner that preserves the \emph{fibration structure}
of the boundary (i.e.\ is local in $\varphi_0$ above) and that
preserves the sign of $\tau$, the dual variable to $t$.

\subsection{Existence of microlocally forward solutions}

Forward solvability of an equation such as $\Box_gu=f$ is a thorny
problem.  As noted in \cite{MoWu:21}, the single-mode version of this
equation is elliptic in the region $\{r<\C\}$.  Suppose we could solve,
e.g., $\Box_gu=\delta_{t_0,r_0} e^{ik\varphi}$ with $r_0>\C$ and with
$u$ supported in $t>T_0$ for some $T_0 \in \RR.$ Then 
by unique
continuation for elliptic equations, $u$ would vanish identically in $\{r<\C\}$.
By the proof of Lemma 6 of \cite{MoWu:21} (with the zero RHS of the equation
used there replaced by $\delta$), $u$ would then further vanish identically for
$r<r_0$, contradicting propagation of singularities in the region
$\{r>\C\}$, where the equation is hyperbolic.

Thus we should not, in general, seek solutions supported in any
set of the form $\{t>T_0\}$: we can at best hope for a weaker notion of
forward solution than that obtained by constraining the support in $t$.

To put the difficulty differently, we do not have a natural global energy estimate for solutions to
$Pu=f,$ as the conserved energy associated with the Killing vector
field $\pa_t$ has mixed sign once we allow the support of $u$ to
overlap $r<\C.$  Thus, the microlocal estimate of
Theorem~\ref{theorem:firstpropagation} above is our only available
energy estimate near $r=0$.  As
energy estimates for adjoint operators usually result in existence
results, this does allow us to prove an existence theorem for forward
solutions to $Pu=f,$ provided we interpret the forward character
\emph{microlocally}: we can characterize the wavefront set of the solution
$u$ as the forward-in-time flowout of the wavefront set of $f$ within
the characteristic set, \emph{together with the forward flowout of the
  string.}  Here ``forward'' and ``backward'' still make sense,
despite the fact that $t$ is not monotone along the null
bicharacteristics which do not reach the string, owing to the fact that asymptotically all
bicharacteristics that do not arrive at $r=0$ escape to the region
$r>\C$ where $t$ becomes monotone along the flow (see Section \ref{sectionBichars} for details).

In order to make the hypotheses on the inhomogeneity work uniformly
down to $r=0,$ regularity statements here involve the \emph{b-Sobolev
spaces} $H_b^m,$ where for $m$ a positive integer, membership in this
space means that applying up to $m$-fold products of the vector fields $r
\pa_r, \pa_\varphi, \pa_t$ (with uniformly bounded coefficients) leaves a function in $L^2.$ Let
$\Phi_+$ denote the (asymptotically) forward-in-time bicharacteristic
flow, over $\{r>0\},$ and let $\Sigma$ denote the characteristic set
of $\Box_g.$
\begin{theorem}\label{theorem:firstforward}
Assume that the perturbation $\pert,$ given by \eqref{pert}, is analytic or else commutes
with both $\pa_t$ and $\pa_\varphi.$
Given compact sets $K_0\subset K \subset X$ with  $K_0 \subset K^\circ$,
if
$f \in H_b^{m}(X)$ with $\supp f \subset K_0$,
there exists $u \in H_b^{m+1}(K^\circ)$ with
$$
Pu=f\quad \text{ on } K^\circ,
$$
such that over $K^\circ \cap\{r>0\},$
\begin{equation}\label{WFrelation}
\WF (u)\backslash \WF(f) \subset \fcal_O\cup \Phi_+ (\WF(f)\cap \Sigma).
\end{equation}
The solution $u$ is unique modulo a distribution $w$ with
$\WF (w) \subset \fcal _O.$
\end{theorem}
A more general version of this result that does not entail such strong
hypotheses on $\pert$ appears below as Theorem~\ref{theorem:forward}.
Theorem~\ref{theorem:forward} also specifies the
microlocal regularity of the solution at $r=0$ and allows for an extension to non-integer Sobolev regularity.

Thus, ``forward'' solutions exist semi-globally (that is to say, over any desired
compact set) for any inhomogeneity $f,$ but the wavefront set of the
resulting solutions may contain wavefront set propagating forward in
time that emanates from the string (i.e., from $r=0$) \emph{at times prior to 
  the support of $f$}.  The string thus may emit information about
disturbances that are yet to occur.

\subsection{Prior work}

The rotating cosmic string metric was introduced by Deser--Jackiw--'t
Hooft in \cite{DeJaTh:84}, but the literature on the behavior of waves
on this background seems to have been little studied.  Our
investigation of solvability of wave equations on this non-causal
background owes a considerable debt to the pioneering work of Bachelot
\cite{Ba:02}, who has obtained a number of results about existence and
uniqueness of
solutions to the wave equation as well as addressing problems in
scattering theory.  Bachelot's results do not apply to the
metric here owing to the singularity; our focus on forward solutions
is likewise a different direction of investigation pursued in
\cite{Ba:02}.  Our emphasis on a \emph{microlocally} causal solution
rather than one whose support lies forward of the inhomogeneity is
partly inspired by the celebrated discussion of global parametrices in
\cite[Chapter 6]{DuHo:72}.

The authors' previous work \cite{MoWu:21} address the
problem of obtaining single-mode solutions to $\Box_g u=f$, i.e.,
solutions of the form $u(t,r) e^{ik \varphi}$.  The mode-by-mode
equation changes type across the cylinder $r=\C$, and turns out to be
of \emph{Tricomi type} and hence amenable to some known
microlocal tools, following previous work of Payne \cite{Pa:98} as
well as the methods of Bachelot \cite{Ba:02}.  In this reduced
equation the singularity at $r=0$ occurs in the elliptic region, hence
we were able to deal with it using relatively standard methods.  Here,
by contrast, singularities can propagate down to (and through) $r=0$,
and the more sophisticated tools of the b-pseudodifferential calculus
are needed.

\subsection{Methods of proof}

The necessary tools for showing microlocal energy estimates that
propagate through the string at $\pa X \equiv \{r=0\}$ are positive commutator methods
in an appropriately adapted pseudodifferential calculus.  Here we use
a version of Melrose's \emph{b-calculus} \cite{Me:93}, but with some
important modifications.  First, the noncompactness of the fibration
of $\pa X$ means that we need a calculus with uniform estimates in the
noncompact directions.  Such a calculus is, fortunately, essentially
described already in \cite{Ho:07}, and in an appendix below we
describe the necessary changes and the translation of the results of
\cite{Ho:07} to our setting; the main properties of the calculus are
summarized in Section~\ref{sec:bcalc}.  More seriously, though, the wave
operator $\Box_g$ does not lie in this calculus: it has singular terms
that we will identify below as squares of  \emph{singular edge vector
  fields}.  The relationship between these singular vector fields,
characterized below in terms of their tangency to the boundary
fibration, and the b-calculus, is discussed in
Sections~\ref{sec:bcalc} and \ref{sec:verybasic}.  In the latter
section, we introduce a further sub-calculus of the b-calculus, the
\emph{fiber-invariant} operators, which have improved commutator properties with the
singular edge operators composing $\Box_g.$  This is the calculus from
which we choose test operators, and with respect to which we define
our wavefront set over $\pa X$.

The estimates needed for the propagation of singularities---or,
equivalently, propagation of regularity---are then
split into the elliptic estimates at $\pa X$
(Section~\ref{sec:elliptic}) and the propagation estimate itself
(Section~\ref{sec:propagation}).  These suffice to establish
propagation of regularity, globally in the fibers of $\pa X.$  These
arguments are parallel to those employed in \cite{MeVaWu:08} to
establish propagation of singularities for the wave equation on
manifolds with edge singularities, but the setup must be modified here 
owing to the noncompactness of the fibers of the boundary fibration.

In Section~\ref{sec:causal} we show existence of microlocally forward
solutions to the wave equation.  Here the essential tool is a variant
of the argument used by Duistermaat--H\"ormander \cite{DuHo:72} in
solving equations of real principal type: the fact that singularities
propagate along rays that escape any compact region allows us to get
lower bounds for the adjoint operator acting on compactly supported
test functions, which then translates into an existence theorem for
the distributional equation.  We ensure that the solution is a
\emph{forward} one by replacing our original operator by one with a
complex absorbing potential, which implies that no singularities can
be arriving in the (arbitrary) compact region in which we are trying
to solve the equation.

\section{Geometric Setting} 

Let $\strng\subset\RR^3_{t,x_1,x_2}$ denote the subset $\{x_1=x_2=0,\
t \in \RR\}.$
Let $X=[\RR^3; \strng]$ be the 3-dimensional manifold obtained
by \emph{blowing up} $\strng,$ i.e.\ simply by replacing $x \in
  \RR^2$ with the polar
coordinates $(r,\varphi) \in [0, \infty) \times S^1.$ We equip the interior
$X^\circ$ with the Lorentizan metric given by
\begin{equation} \label{metric}
	\begin{split}
		g &= -dt^2+(r^2-\C^2)d\varphi^2 +2\C dtd\varphi +dr^2\\
		&= -(dt-\C d\varphi)^2 + r^2 d\varphi^2 + dr^2.
	\end{split}
\end{equation}
Note that $\pa X$ is naturally equipped with a fibration
compatible with the metric: define the (complex) vector field $$\fib \equiv i^{-1}(\C 
\pa_t+\pa_\varphi),$$ and that the integral curves of the real vector field $i\fib$
at $r=0$ are the (helical) fibers of a fibration of the cylinder
$$\pa X =\{r=0,
\varphi \in S^1, t \in \RR.\}$$  Let $\pi_0$ denote the projection map
$\pa X \to S^1$ given (somewhat non-canonically) by \begin{equation}\label{pizero}\pi_0\colon (r=0, \varphi,t) \mapsto
	\varphi-t/\C  \bmod 2 \pi\ZZ.\end{equation}  Thus $\pi_0$ maps each point in $\pa X$ to
      the point in the same fiber over $\{t=0\}.$
The rotating cosmic string metric then takes the special form $g=\pi_0^* g_0 + r^2 h$ where
$h$ is a symmetric two-form in $t,\varphi.$  In particular
the only nontrival components in the fiber directions are $O(r^2).$

\subsection{Geometry of bicharacteristics} \label{sectionBichars}
The metric $g$ is not globally hyperbolic: the parametrized closed curve
$$
\big\{(r=r_0, \varphi=s, t=t_0): s \in [0, 2\pi]\big\}
$$
is timelike if $r_0<\C.$  There are, however, no closed causal
geodesics, as will become clear below.

Using usual dual variables in the cotangent bundle (which
will be replaced later on with the better-adapted fiber variables in the
b-cotangent bundle), the
symbol of $\Box_g$ is
$$
p= \tau^2-\xi^2-\frac{(\C \tau+\eta)^2}{r^2},
$$
and has Hamilton vector field
$$
\big( 2 \tau-2r^{-2}\C(\C \tau+\eta)\big)\pa_t-2\xi \pa_r-2r^{-2} (\C
\tau+\eta) \pa_\varphi -2 r^{-3} (\C \tau+\eta)^2 \pa_\xi.
$$
On the characteristic set $\{p=0\}$, the coefficient of $\pa_t$ has
fixed sign as $\tau$ as long as $r>\abs{\C}.$  Thus the $t$ variable is
monotone on the null bicharacteristics as long as they remain in $r>\abs{\C}.$

Only special bicharacteristics reach $r=0,$ just as would be the case for
the Minkowski metric in cylindrical coordinates.  In Minkowski space,
the necessary and sufficient condition would be vanishing of angular
momentum, but here it is the condition
$$
\C \tau+\eta=0,
$$
which is manifestly conserved along the flow. The time variable is
monotone along each of these curves, with $\dot t = 2 \tau;$ moreover
$dr/dt=  -\xi/\tau=\pm 1$ on the characteristic set.

That the metric is in fact flat away from $r=0$ is easily seen via the
(local in $\varphi$) change of variables
$$
t'\equiv t-\C \varphi,
$$
which reduces the metric to the Minkowski metric
$$
g=dr^2 + r^2 d\varphi^2-dt'^2;
$$
projections of null bicharacteristics then become (forward and
backward) Minkowski geodesics,
expressed in cylindrical coordinates
\begin{equation}\label{flatgeodesic}
(r\cos \varphi, r \sin \varphi, t') = (x_0+v_x s, y_0+v_y s, t_0\pm s),
\end{equation}
with $v_x^2+v_y^2=1.$  (Note that such a bicharacteristic stays within
the coordinate patch we have introduced here, effectively by
introducing a branch cut in the $xy$-plane, since $\varphi$
asymptotically increments by
$\pm \pi$ under the flow along a Minkowski null-bicharacteristic.)
In these coordinates, then, it is trivial to see that
every null bicharacteristic (except those hitting $r=0,$ which we
exclude from discussion for now) escapes the region $\{r\leq \abs{\C}\},$
hence, returning to our original coordinate system, we see that our
original $t$ variable is
eventually monotone at both ends of every null bicharacteristic curve,
with $\dot t$ having consistent sign at both ends.  We can thus orient
each bicharacteristic curve in a direction that makes $\dot t$
positive on both ends.

The non-monotocity of $t$ is moreover limited: the coordinate $t'$ is
monotone owing to the Minkowski geometry \eqref{flatgeodesic}.  Since
$t=t'+\C \varphi,$  letting $t(s)$, $t'(s)$ and $\varphi(s)$ denote the values along
the bicharacteristic shows that
$$
t(s)-t(0)=t'(s)-t'(0) + \C (\varphi(s)-\varphi(0));
$$
since, as noted above, the variation in $\varphi$ along a null
geodesic is $\pm \pi,$ if we choose signs such that $\dot t=+1$
asymptotically (i.e., $t(s) \to +\infty$ as $s \to +\infty$), then
$t(s)-t(0)$ can never be less than $-\abs{\C}\pi.$

A geodesic aimed nearly at the string ($r=0$) shows the
non-mono\-ton\-icity of $t$ most strikingly.  As before we can choose the
sign of our pa\-ra\-met\-ri\-za\-tion to arrange
$t'(s)-t'(0)= s$.  If the geodesic passes very close to $r=0$ then
consideration of lines in $\RR^2$ shows that $\varphi$ is approximately constant except at the moment when it
passes by $r=0$---without loss of generality, at time $s=0$---when it rapidly increments or decrements (depending on whether it leaves the string to the left or to the right) by
$\pi-\ep$ in a short interval $s \in [-\delta, \delta].$  Thus,
$$\begin{aligned}t(\delta)-t(-\delta) &= t'(\delta)-t'(-\delta)
+\C(\varphi(\delta)-\varphi(-\delta))\\ &=2\delta \pm\C(\pi-\ep)\\ &\approx \pm\C \pi.\end{aligned}$$  Hence in the limit in which such geodesics pass through
the spatial origin $r=0,$ the time variable instantaneously increments
or decrements by $\C \pi$ at the moment of interaction with the
string, and is otherwise continuous and monotone increasing.

\section{b-Geometry and the b-Calculus}\label{sec:bcalc}

In this section, we describe the geometric setting of the
``b-category'' and its associated analytic objects as espoused by
Melrose in, e.g., \cite{Me:93} (but with a slight complication in its
application in the case at hand
involving uniformity as $\abs{t} \to \infty$).

The set of b-vectors fields on $X$ is:
$$
\V_b(X) = \{\text{smooth vector fields tangent to }\pa X\}.
$$
Thus,
$$
\V_b(X) = \mathcal{C}^\infty(X)\text{-span} (r\pa_r, \pa_t, \pa_\varphi).
$$
We then define the b-differential operators as sums of products of
these vector fields:
$$
\Diff_b^k(X) = \left\{\sum_{l, k'(l) \leq k} a_{l}(r,t,\varphi)
   V_{1,l}\dots V_{k'(l),l}, \quad 
   a_{l} \in \CI, V_\bullet \in \V_b(X)\right\}.
$$
Owing to the noncompactness of $X$, however, we will employ versions
of these spaces involving uniform estimates: let $\CI_u(X)$ denote the
space of $\CI$ functions of $r, t,\varphi$ with uniform derivative
bounds:
$$
f \in \CI_u \Longleftrightarrow  \pa_r^{i} \pa_t^{j}
\pa_\varphi^{k} f \in L^\infty(X)\text{ for all } i,j,k \in \NN.
$$
Let
$$
\V_\bu(X)= \CI_u(X)\text{-span} (r\pa_r, \pa_t, \pa_\varphi),
$$
and
$$
\Diff_\bu^k(X) = \left\{\sum_{l, k'(l) \leq k} a_{l}(r,t,\varphi)
   V_{1,l}\dots V_{k'(l),l}, \quad 
   a_{l} \in \CI_u, V_\bullet \in \V_\bu(X)\right\}.
$$
We also require a weighted version of this space, denoted $r^\ell \Diff_\bu^m(M).$

The vector fields in $\V_\bu(X)$ are sections of a vector bundle, denoted
$\Tb X$.  The dual bundle, denoted $\Tbstar X,$ is the bundle whose local
sections are $1$-forms spanned over $\CI(X)$ by $dr/r,\ dt,\ d\varphi.$
The \emph{principal symbol} of a b-differential operator is a
polynomial function on $\Tbstar X:$ if
$$
A= \sum_{i+j+k \leq m} a_{ijk}(r,t,\varphi) (rD_r)^i D_t^j D_\varphi^k,
$$
$$
\sigma^m_b(A) = \sum_{i+j+k = m} a_{ijk}(r,t,\varphi) \xib^i \taub^j \etab^k
$$
where the variables $\xib,\ \taub,\ \etab$ are defined by 
writing the canonical one-form on $\Tbstar X$ as
$$
\xib \frac{dr}r+\taub dt + \etab d\varphi.
$$
Here and throughout the rest of the paper we use $\xib, \taub,$ and
$\etab$ to be fiber variables in the b-cotangent bundle.  We are also
employing the usual convention in microlocal analysis that
$$
D_z=i^{-1} \partial_z
$$
for any coordinate $z$.

Note that
there is a symplectic form on $\Tbstar X^\circ$ defined by the
differential of the canonical one-form
\begin{equation}\label{bsymplectic}
\omega \equiv d\xib\wedge \frac{dr}r+d\taub \wedge dt + d\etab \wedge d\varphi.
\end{equation}
hence there is an associated notion of Hamilton vector fields.

We will have occasion to employ homogeneous coordinates on $\Tbstar
X$ in making symbol constructions; to this end, set
\begin{equation}\label{hatcoordinates}
{\xibhat}=\frac{\xib}{\abs{\taub}},\
{\etabhat}=\frac{\etab}{\abs{\taub}}, \taubhat=\frac{\taub}{\abs{\taub}}
\end{equation}
The coordinates $\xibhat,\ \etabhat$, together with $\taub$ itself,
are a coordinate
system in the fibers except at $\taub=0;$ as we will see below, this set is disjoint
from the characteristic set of $\Box_g$.  The function $\taubhat$ is of course
just $\pm 1$-valued, but is useful notation nonetheless.

We can write $\Box_g$ with respect to the b-vector fields as
	\[ \Box_g = -\left(1-\frac{\C ^2}{r^2}\right) \partial_t^2 +
          \frac{1}{r^2}\big((r\pa_r)^2 + \pa_\varphi^2\big) + \frac{2\C }{r^2}
          \partial_t\partial_\varphi \in r^{-2} \Diff_\bu^2(X)\]
so that the principal symbol associated to $\Box_g$ is given by
\begin{equation}\label{bsymbol}
		p(t,r,\varphi, \taub, \xib, \etab) 
			= \taub^2 - r^{-2} \xib^2 - r^{-2}\left( \C \taub +\etab\right)^2.
                      \end{equation}
	Thus the characteristic set is $\Sigma = \{ r^2 \taub^2 = \xib^2 + (\C \taub + \etab)^2\}$. Over the boundary $\{r=0\}$ we have $\Sigma|_{\{r=0\}} = \{ \xib = \C \taub + \etab = 0 \}$.
	
	The b-Hamilton vector field of $p$ is given by
	\begin{align*}
		^bH_p &= \pa_{\taub} p\, \pa_t + r\pa_{\xib} p\, \pa_r + \pa_{\etab} p\, \pa_\varphi - \pa_t p\, \pa_{\taub} - r\pa_r p\, \pa_{\xib} - \pa_\varphi p\, \pa_{\etab} \\
			&= \big(2\taub - \frac{2\C }{r^2}(\C \taub+\etab)\big)\pa_t -\frac{2\xib}{r} \pa_r - \frac{2}{r^2}(\C \taub+\etab) \pa_\varphi - 2\left(\frac{\xib^2 + (\C \taub+ \etab)^2}{r^2}\right)\pa_{\xib},
	\end{align*}
which yields the (rescaled) b-Hamilton vector field, itself in
$\V_b(\Tbstar X)$, given by
\begin{equation}\label{bHam}
\begin{aligned}
		\frac{r^2}{2}\hbox{ } ^bH_p = (r^2\taub -\C (\C \taub+
          \etab))\pa_t - \xib r \pa_r -(\C \taub+ \etab) \pa_\varphi -(\xib^2+(\C \taub+\etab)^2) \pa_{\xib}.
	\end{aligned}\end{equation}
        Let
$$\Phib(s)\equiv \exp (s(r^2/2) \hbox{}^bH_p)$$ denote the flow generated by
this rescaled vector field.
        
        	Looking at $\frac{r^2}{2}\hbox{ }^bH_p$ over the
                characteristic set at $r=0$, we see
                $\frac{r^2}{2}^bH_p|_{\Sigma|_{\{r=0\}}} = 0$ so we
               cannot expect to obtain any propagation at $r=0$ by
               invocation of standard propagation of singularities results. 
                \begin{remark}\label{remark:flowout}
If a null bicharacteristic, i.e., an integral curve of
$(r^2/2)\hamvf_p$ inside the characteristic set,
parametrized by $s,$ approaches $r=0$ as $s \to \pm\infty$ then
certainly $\C \taub + \etab=0$ since $p,\taub,\etab$ are all conserved.
Conversely, if $\C \taub+\etab=0$ then solving in $r,\xib$ shows that in
fact $r\to 0$ as $s \to \pm\infty$ (with the $\pm$ determined by the
signs of $\xib$ and $\taub$).  Thus $\C \taub+\etab=0$ exactly
defines the flowout of $r=0$ inside $\Sigma.$

                  \end{remark}
\subsection{Uniform b-calculus}\label{sectionbPseuds}
The b-calculus of pseudodifferential operators is a microlocalization
of the b-differential operators.  Here we exclusively employ a \emph{uniform} version of the b-calculus adapted to the
noncompact but quasi-periodic setting under study.

A function $a$ on $\Tbstar X$ is in the symbol class
$S_{\unif}^{m}(\Tbstar X)$ if for all multiindices $\alpha,\beta$ and
integers $N \ge 0$ it satisfies estimates
\begin{equation}\label{Sunif}
\big \lvert \pa_{r,t,\varphi}^\alpha \pa_{\xib,\taub,\etab}^\beta a\big \rvert \leq C \ang{(\xib,\taub,\etab)}^{m-\abs{\beta}}(1+r)^{-N},
\end{equation}
\emph{where the estimate is assumed to hold over all of $\Tbstar X,$}
noncompactness of $X$ notwithstanding.  In particular, the estimate is
to be uniform as $\abs{t}\to \infty$ (the radial variable in practice
will only range over a compact set).

In this section we outline the development of the b-pseudodifferential calculus built by quantizing symbols in $S_{\unif}^m(\Tbstar X)$. A more detailed account can be found in Appendix~\ref{appendix:bcalculus}, where we follow an alternate development of the calculus from the slightly different point of view
used in \cite{Ho:07}, which is in fact equivalent to the account given here (as proved in Proposition~\ref{prop:equivalent}).

We may quantize symbols in $S_{\unif}^m(\Tbstar X)$ to obtain
operators with the following Schwartz kernel (Schwartz kernel of
an operator is denoted $\kappa(\bullet)$):
\begin{equation}\label{Opb.v1}
	\begin{split}
\kappa&(\tOpb(a)) \\
	&=\int e^{i [(r-r')\xi/r
  +(t-t')\tau + (\varphi-\varphi') \eta]} \chi a(r,t,\varphi, \xi, \tau,
  \eta) \frac{r'}{r} \, d\xi d\tau d\eta \abs{dt'\, d\varphi' \, dr'/r'}.
 	\end{split}
\end{equation}
Here we have included a half density factor that is both appropriate
to the geometry of the half-line and also which arises naturally from
the construction of the b-calculus by quantization of certain singular
symbols in \cite{Ho:07} discussed below;
the function $\chi$ is a cutoff in the radial variables:
$\chi=\chi(r'/r),$  where $\chi(s)$ is supported near $s=1.$
The convention here is that functions to which this operator is
applied should be viewed as $2 \pi$-periodic functions in the
$\varphi$ variable, and the symbol $a$ is likewise periodic in
$\varphi.$  The application of the operator to a function is then
integration over $\RR_\varphi,$ with the result being again periodic
(cf.\ \cite[Section 5.3.1]{Zw:12}).

We have employed the notation
$\tOpb$ to distinguish the quantization used here from that in
Appendix~\ref{appendix:bcalculus}, which may only be applied to a
certain subclass class of ``lacunary'' symbols.

\begin{remark}
	We recall from \cite[Chapter 4]{Me:93} that it is illuminating to view
	the Schwartz kernels of these operators in the ``b-double-space''
	obtained from real blowup $X^2_b=[X\times X ; (\pa X)^2]$, which here
	corresponds to just replacing $r,r'$ variables in $[0,\infty)^2$ with polar coordinates
	in the quarter-plane; equivalently we may use the simpler
        substitutes for polar coordinates given by
        $$
(\rho,\theta) \equiv \big(r+r', \frac{r-r'}{r+r'})
        $$
with $\rho \in [0, \infty)$ and $\theta \in [-1,1]$.        
It will often be valuable to use the even simpler coordinates $r$ and
$s \equiv r'/r$, but then we must recall that as $\theta \to -1,$
$s\to +\infty$, hence these coordinates are not quite global.
The set $\{\rho=0\}$ is the ``front face'' of the blowup, the new boundary face
	that we have introduced to replace the corner $r=r'=0;$ the ``side
	faces'' are now $\theta=\pm 1,$ a.k.a., $s=0$ resp.\ $s=+\infty.$ The operators in the usual
	b-calculus are those whose Schwartz kernels are conormal to the lift of the diagonal to
	this space, smooth up to the front face, and rapidly decaying at the
	side faces.  The operators considered here have the further feature of
	enjoying uniform estimates in the noncompact boundary variables $t,t'.$
\end{remark}

Operators of the form \eqref{Opb.v1} are not quite a rich enough class to form a
calculus closed under composition, owing to the presence of the cutoff
$\chi.$  In general, composition causes the support of Schwartz kernels to spread transverse to the diagonal. Thus in order to make the calculus closed under composition, we need
to allow residual operators that have Schwartz behavior in this
transverse direction. 

\begin{definition}
	We say an operator $R$ is residual if its Schwartz kernel satisfies the following estimate in coordinates on $X_b^2$ given by $\rho=r+r',$ $\theta=(r-r')/(r+r')$: for all $\alpha,\beta,\gamma, N,$
	$$
	\big\lvert \partial_{t,\varphi, t', \varphi'}^\alpha \partial_\rho^\beta
	\partial_\theta^\gamma \rho \kappa(R)\big\rvert \leq C_{\alpha,\beta,\gamma,N}
	(1+\abs{t-t'}+ \rho)^{-N}(1-\theta^2)^N.
	$$
\end{definition}
Note that the $(1-\theta^2)^N$ factor has the effect of enforcing
rapid decay of the kernel on the side-faces ($\theta=\pm 1$) of the
resolved space $X^2_b.$

With our definition of residual operators in hand, we may define our uniform b-pseudodifferential operators.

\begin{definition}
	An operator $A$ is in $\Psib^m$ if $A = A_0 +R$ for some $A_0 = \tOpb(a)$ with $a \in S_{\unif}^m(\Tbstar X)$ and $R$ a residual operator.
\end{definition}
That this is equivalent to the different definition introduced in
Appendix~\ref{appendix:bcalculus} is the content of Proposition~\ref{prop:equivalent}.



  \begin{proposition}
The space $\Psib^\star (X)$ is a calculus, i.e., a filtered
$\star$-algebra, and $\Psib^0(X)$ is bounded on $L^2(X)$.  There is a
principal symbol map $\sigmab{s} : \Psib^s(X) \to S_{\unif}^s(X)/S_{\unif}^{s-1}(X)$ that yields a
short exact sequence
$$
0 \to \Psib^{s-1}(X) \stackrel{\text{incl.}}{\to} \Psib^s(X) \stackrel{\sigmab{s}}{\to} S_{\unif}^s(X) \to 0.
$$

Furthermore if $A \in \Psib^s(X)$, $B \in \Psib^k(X)$, then $[A,B] \in \Psib^{s+k-1}$ satisfies
$$ \sigmab{s+k-1}([A,B]) = -i\left\{ \sigmab{s}(A), \sigmab{k}(B) \right\} $$
with $\{ \cdot, \cdot \}$ the Poisson bracket defined using the
symplectic form \eqref{bsymplectic}.
    \end{proposition}
The proof of this proposition is the content of
Appendix~\ref{appendix:bcalculus}; essentially the whole result
is obtainable from results in \cite[Chapter 18.3]{Ho:07}. 

The principal symbol map of the b-calculus is the obstruction to an
operator being lower order, but not the complete obstruction to
compactness, which is also governed by behavior at $r=0.$ for a
b-differential operator
$$
A=\sum a_{ijk} (r,t,\varphi) (rD_r)^i D_t^j D_\varphi^k
$$
we can create a scaling (in $r$)-invariant operator by freezing
coefficients at $r=0;$ this is the \emph{indicial operator}
$$
I(A) \equiv \sum a_{ijk} (0,t,\varphi) (rD_r)^i D_t^j D_\varphi^k.
$$
In \cite[Section 4.15]{Me:93}, the extension of this operator to
b-pseudodifferential is discussed, and this same discussion applies
with no change here: there is a map
$$
I: \Psib^s(X) \to \Psi^s_{bu, I}(X)
$$
with the latter being the space of operators in the calculus that
commute with scaling in the $r$ variable.  The essential features here
are as follows:
\begin{proposition}\label{prop:indicial}
 $I$ is an algebra homomorphism, with the property that for $A \in \Psib^s(X),$ 
$$
I(A)=0 \Longleftrightarrow A \in r\Psib^s(X)\Longleftrightarrow A \in
\Psib^s(X) r,
$$
\end{proposition}

We record some useful consequences involving commutators with $r$ and $rD_r.$
\begin{lemma}\label{lemma:indicial}
  Let $A \in \Psib^s(X).$  Then
\begin{equation}\label{i1}
[r, A] \in r\Psib^{s-1}(X)=\Psib^{s-1}(X) r,
\end{equation}
\begin{equation}\label{i2}
r^{-1}Ar,\ rA r^{-1} \in \Psib^{s}(X),\quad \sigmab{s}(r^{-1}Ar)=\sigmab{s} (rAr^{-1})=\sigmab{s}(A).
  \end{equation}
  \begin{equation}\label{i3}
[rD_r, A] \in r \Psib^{s-1}(X).
    \end{equation}
  \end{lemma}
  \begin{proof}
Since $I(r)=0,$ \eqref{i1} follows immediately.  Then \eqref{i2}
follows since, e.g., $r^{-1} Ar=A-r^{-1} [r,A]$.  Finally, \eqref{i3}
holds since $I(rD_r)=rD_r$ commutes with $I(A)$ by scaling invariance of
the latter.
    \end{proof}

\begin{remark} We will need notions of operator wavefront set (a.k.a.\
  microsupport) and the related wavefront set of distributions
  associated with the b-calculus.  The uniform versions of these
  notions that we need, however, are slighly unsatisfactory, as
  uniform estimates on a symbol (with rapid fiber decay) in a \emph{noncompact} open conic set $\Omega
  \subset \Tbstar X$ are of course not equivalent to the validity of
  the estimate locally in a conic neighborhood of each point in
  $\Omega.$  Without the addition of some notion of ``microsupport at
  infinity,''  therefore, it would not suffice for us to define the
  microsupport as a set.  We therefore postpone discussions of
  wavefront sets to the discussion of the \emph{fiber-invariant calculus}
  below, where taking fiber quotients will supply the desired uniform
  estimates in a more satisfying way.\end{remark}



\subsection{Edge structure}
	A manifold $M$ with boundary is endowed with an edge structure if the boundary $\pa M$ admits a fibration:
        \[ Z \rightarrow \pa M \xrightarrow{\pi_0} Y \] with fiber $Z$
        and base $Y$ (see \cite{Ma:91}). The \emph{edge vector fields} are defined to be
        those vector fields that are tangent to $\pa M$ and
        additionally to the fibers within $\pa M.$ In coordinates
        adapted to the fibration, with $x$ a boundary defining
        function and $(y,z)$ coordinates on $\pa M$ associated to a
        local trivialization of the fibration so that
        the $y$ variables are constant on fibers, we have
		\[ \mathcal{V}_e(M) = \mathcal{C}^\infty(M)\text{-span} ( x\pa_x, x\pa_y, \pa_z)\]
		The edge vector fields are the smooth sections of the edge tangent bundle, which we denote $\Te M$. The edge cotangent bundle, denoted $\Testar
        M$, is then defined to be dual to $\Te M$, with dual one forms $\frac{dx}{x}, \frac{dy}{x},$ and $dz$. 

        Our setting has an edge structure given by the fibration
        tangent to the vector field $\fib;$ here the leaves are
   all diffeomorphic to $\RR$ and the leaf space is smooth and may be
   identified with $S^1$ by the map $\pi_0$ in \eqref{pizero}.  This
   map is of course somewhat non-canonical: it could be regarded as taking the
  unique point in the circle in each leaf at time $t=0.$  When possible we will
 employ the more invariant terminology
 $$
p\sim q
$$
to mean that $p,q\in \pa X$ are in the same leaf.

 In coordinates, we then obtain edge forms and vector fields as follows.
	 Set $x = r, y = t- \C \varphi,$ and $z = \varphi$. 
	 Then the sections of $\Te X$ are the $\mathcal{C}^\infty(X)$ span of
		\[ r \pa_r, \quad r \pa_t, \quad \hbox{ and }\C \pa_t+\pa_\varphi.\]
	 The canonical one form on $\Te^*X$ is then
	 	\begin{equation}\label{canonical1} \sigma = \xie \frac{dr}{r} + \taue \frac{dt-\C d\varphi}{r} + \etae d\varphi. \end{equation}
		

                We remark that we could alternatively have used
                $$
x=r,\ y=t-\C \varphi,\ z=t
                $$
as our coordinates, i.e., we could have used $t$ as a fiber
coordinate rather than $\varphi.$  (This has the virtue that the
noncompactness of the helical fibers is more immediately apparent, but doesn't
make much difference.)  With these choices, the edge vector fields
become
$$
r\pa_r,\ -\frac r\C  \pa_\varphi,\ \pa_t+ \frac 1\C  \pa\varphi,
$$
which are easily seen to give an alternative basis for the edge vector fields defined above.

In this paper the edge structure is important for understanding the
elliptic set over the boundary for $\Box_g$, which can be regarded
as a weighted edge differential operator, but the ``edge calculus'' of
Mazzeo \cite{Ma:91} plays no direct role here.

\subsection{Edge/b relationship}\label{sectionCompressedBundle}

It is of considerable motivational and practical interest to
understand the relationship between the b and edge cotangent bundles
and characteristic sets.

We define the map $\pi: \Testar X \to \Tbstar X$ via the
inclusion map $r\Testar X \xhookrightarrow{} \Tbstar X$. That is, for
$q \in \Testar X$ with coordinates given by the canonical one-form \eqref{canonical1},
we have
$rq = \xie dr + \taue (dt -\C d\varphi) + r\etae d\varphi$ which can be
written as an element of $\Tbstar X$ as
$(r\xie) \frac{dr}{r} + \taue dt + (r\etae - \taue \C )d\varphi$. Thus
\[ {\pi}: (r,t,\varphi, \xie, \taue, \etae) \mapsto ( r,t, \varphi,
  r \xie, \taue, r\etae - \C \taue)=(r,t,\varphi, \xib, \taub, \etab). \]
Note that for $q \in \pi(\Testar X) $ we have
\begin{equation}\label{qindotpi}
	\xib =\etab + \C \taub = 0
\end{equation}
at the boundary $\{ r = 0\}$.

We define a further map which combines $\pi$ with the quotient by
fibers.  First, note that the screw-displacement flow along the vector
field $i\fib,$ extended trivially to the interior of $X$ (the
choice of extension turns out to be irrelevant) preserves the
fibration of $\pa X$, hence gives an
identification of $\Testar_pX$ and $\Testar_qX$ whenever $p,q \in \pa
X$ with $p \sim q;$
it also identifies $\Tbstar_pX$ and $\Tbstar_qX,$ and commutes with
the map $\pi,$ so that we gain an equivalence relation (also denoted
$\sim$) on $\pi(\Testar_{\pa X} X).$ 

Concretely, we may employ $(t,\varphi,
\taub)$ as coordinates on $\pi(\Testar_{\pa X} X),$ since $\xib=0$ there and
$\etab$ is determined by \eqref{qindotpi}; then 
$$(t,\varphi,\taub) \sim (t',\varphi',\taub') \Longleftrightarrow
\taub=\taub',\ \varphi-t/\C=\varphi'-t'/\C\bmod 2 \pi\ZZ.$$
Finally, we let the \emph{compressed b-cotangent bundle} be the quotient
$$
\Tbdotstar X=\pi(\Testar X)/{\sim};
$$
Since $\pi$ commutes with the flow along $i\fib,$ we may (and will) view this as a subset of $\Tbstar X/{\sim},$ and give it the
subspace topology of that space (which itself has a quotient topology).
There is a
natural map
$$
\dot{\pi}\colon \Testar_{\pa X} X \to \Tbdotstar_{\pa X} X
$$
given by mapping a point to the equivalence class  of its image under $\pi$ (i.e. $\dot{\pi}(q) = [\pi(q)]_{\sim}$).
In coordinates, a point in $\Tbdotstar_{\pa X} X$ is simply given by the
equivalence class (over varying $s$) of
$$
(t=\C s, r=0, \varphi=\varphi_0 +s\bmod 2\pi \ZZ, \taub=\tau_0, \xib=0, \etab=-\C \taub_0),
$$
with $\varphi_0\in S^1, \tau_0\in \RR \backslash \{0\}$ thus providing
coordinates for
$\Tbdotstar_{\pa X} X.$  We will also have occasion to use the usual
notation
$$
\Sbdotstar X
$$
for the $\RR_+$-quotient of this bundle, which over the boundary is
just $S_\varphi^1 \times S_{\tau}^0$.

Given $\varrho=(\varphi_0, \tau_0) \in \Tbdotstar X$ we let
\begin{multline}
\fcal_{I/O, \varrho} = \{(t,r,\varphi, \taub,\xib, \etab)\in \Sigma: \C
\taub+\etab=0, \sgn(\xib/\taub)=\pm 1,\\  \taub=\tau_0,\ \lim_{s
  \to \pm \infty} (\varphi-t/\C)\circ \Phib(s)=\varphi_0\}
\end{multline}
with $I$ corresponding to $+1$ and $O$ to $-1$ in the sign of
$\xib/\taub,$ and the direction of the limit in $s$ being chosen with
$\pm=-\sgn\xi$.  (Recall that $\Phib(s)$ denotes the flow along the
rescaled b-Hamilton vector field.) These are the points
``incoming'' toward or ``outgoing'' from the fiber indexed by
$(\varphi_0, \tau_0)$, according to whether $dr/dt=-\xib/r\taub$ is negative
resp.\ positive.

\section{Twisted $H^1$ and the fiber-invariant b-calculus}\label{sec:verybasic}

Recall that we define a Sobolev-type norm on the space $\Hone$ by
	\[ \|u\|_\mathcal{H}^2 =  \|D_r u\|_{L^2(X)}^2 + \|r^{-1}\fib u\|_{L^2(X)}^2 + \|D_tu\|_{L^2(X)}^2 + \| D_\varphi u\|_{L^2(X)}^2
          + \| u\|_{L^2(X)}^2. \]
Note that owing to the presence of the singular $r^{-1}\fib\equiv r^{-1} (\C  D_t +
D_\varphi)$ term, we could dispense with either the $D_t$ or the $D_\varphi$
term (but not both).  We define the space $\Hone$ to be the closure of
$\mathcal{C}_c^\infty(X^\circ)$ (i.e.\ vanishing at $\strng$ is
imposed).  Thus $\Hone$ agrees with the usual Sobolev space $H^1$ away from
$r=0,$ but is an appropriately twisted version of $H^1(\RR^3)$ at the cosmic
string itself.

Let $\Hmone\subset \mathcal{C}^{-\infty}(X)$ denote the dual space of $\Hone$ with respect
to the $L^2$ inner product.   Thus
$$
P: \Hone \to \Hmone
$$
is bounded.

\begin{remark}
The uniform b-calculus is not bounded on $\Hone,$ since even a
multiplication operator $M_\psi$ by a cutoff function $\psi(t)$ has
the defect that
$$
\norm{M_\psi u}_{\Hone}^2 \geq \norm{r^{-1} \fib(\psi(t) u)}^2, 
$$
and this yields $r^{-1}u$ terms when $\fib$ acts on $\psi(t)$ which are
not, in general, bounded by $\norm{u}_{\Hone}^2.$  Consequently we
will specialize further to certain operators with better commutation
properties with $\fib.$
  \end{remark}

\begin{definition}
  An operator $A \in \Psib^s(X)$ is said to be \emph{fiber-invariant} if
  $$
[A, \fib] =0.
$$
and we write $\Psibf^s(X)$ to denote the space of the fiber-invariant
b-pseudo\-differ\-en\-tial operators of order $s$.  We let $\Diffbf$
denote the subalgebra of fiber-invariant differential operators.

Let $\CIf$ denote the space of fiber-invariant smooth functions, i.e.\
those annihilated by $\fib.$
\end{definition}
\begin{remark} Operators satisfying weaker notions of
  fiber-invariance, such as the \emph{very basic} operators introduced
  in \cite[Section 10]{MeVaWu:08}, would suffice to obtain boundedness
  on $\Hone$.  However some difficulties involving iterated
  commutators with $\fib$, e.g.\ in the proof of microlocal elliptic
  regularity, seem to make these less rigid operators more difficult to use in the setting
  under consideration here.\end{remark}

Note that fiber-invariance is preserved under composition,
by the derivation property of commutation, and is also preserved under
multiplication by fiber-invariant smooth functions.

\begin{definition}
Let $S_{\F}^m(\Tbstar X)$ denote the space of symbols $a \in
S_\unif^m(\Tbstar X)$ that additionally satisfy
$$
\F(a) =0.
$$
  \end{definition}

  Recall that a convenient way of quantizing arbitrary symbols to operators in
  our calculus is given by the map $\tOpb$ defined in \eqref{Opb.v1}.
\begin{lemma}
Let $A=\tOpb(a) \in \Psib^s(X).$  Then $A \in \Psibf^s(X)$ iff $a \in
S_{\F}^m(\Tbstar X)$.
  \end{lemma}
  \begin{proof}
    We simply note that the Schwartz kernel of $[\fib, A]$ is
   $\tOpb(\F(a)).$ 
\end{proof}

\begin{lemma}\label{lemma:easycomm}
	 Let $A \in \Psib^s(X)$ with $\sigmab{s}(A) = a$.  Then
	 	\[	[D_r,A] = E + FD_r = E' + D_r F' \quad \hbox{ and } \quad [r^{-1},A] = r^{-1}G=G'r^{-1}\]
                for $E, E' \in \Psib^s(X)$ and $F, F', G,G' \in \Psib^{s-1}(X)$, with
                $$
                \begin{aligned}
                  \sigmab{s}(E) &= \sigmab{s}(E') = -i \pa_r(a),\\
                  \sigmab{s-1}(F) &= \sigmab{s-1}(F') = \sigmab{s-1}(G) = \sigmab{s-1}(G')= -i \pa_{\xib}(a).
                  \end{aligned}
                  $$
          If $A$ is additionally
        fiber-invariant, then all the other operators are as well.          
\end{lemma}

\begin{proof}
	First we recall that
	$$
	[r,A]\in r\Psib^{s-1}(X)=\Psib^{s-1}(X) r
	$$
	by Lemma~\ref{lemma:indicial}.
	Moreover,
	$$
	[r,A]= \tilde{G}r
	$$
	with
	$$
	\sigmab{s-1}(\tilde{G}) = -ir^{-1} (\hamvf_r)(a)= i\pa_{\xib}(a).
	$$
	Thus
	\begin{equation}\label{rinvcomm}
	[r^{-1}, A]=-r^{-1}[r,A]r^{-1}  = -r^{-1}\tilde{G} = r^{-1}G
	\end{equation}
	where
	$$ \sigmab{s-1}(G) = - i \pa_{\xib}(a). $$  An analogous
        version of the computation with $[r,A]=r\tilde{G}'$ proves
        that we we may also write this operator as $G'r^{-1}$.
	
	Next we write
	\[ [D_r,A] = r^{-1}[rD_r,A] + [r^{-1},A]rD_r = E + FD_r \]
        where $E = r^{-1}[rD_r,A]$ and $F = [r^{-1},A]r$. By
        \eqref{rinvcomm} we see $F = r^{-1}Gr$ so that
        $\sigmab{s-1}(F) = \sigmab{s-1}(G)$, as desired. The desired
        properties of $E$ follow from the observation
        $[rD_r,A] \in r\Psi_b^s$ by Lemma~\ref{lemma:indicial}.
        
	The alternate second form, $E' + D_rF'$, follows using the above, now writing $D_r = rD_r r^{-1} + r^{-1}$. Indeed,
	\begin{align*}
		[D_r,A] &= [rD_r r^{-1},A] -i [r^{-1},A] \\
			&= (rD_r-i)[r^{-1},A] + [rD_r,A]r^{-1} \\
			&= D_rF' + E'
	\end{align*}
	where $F' = r[r^{-1},A]$ and $E' = [rD_r,A]r^{-1}$.

\end{proof}

Finally, we show that $0$-th order fiber-invariant pseudodifferential operators map $\Hone$ continuously into itself.

\begin{lemma}\label{lemma:mappingprop}
	Let $A \in \Psibf^0$ and $u \in \Hone.$ Then 
	$$ \|Au\|_\Hone \le C\|u\|_\Hone $$
	and $C$ depends only on the seminorms of the symbol of $A$.
\end{lemma}

\begin{proof}
Let $A \in \Psibf^0$ and $u \in \Hone$.  By boundedness of the uniform calculus, all
terms in $\norm{A u}_{\Hone}^2$ are bounded except for $\norm{D_r Au
}^2$ and $\norm{r^{-1}\fib A u}^2.$  Since $[\fib,A]=0$, boundedness of
these two terms additionally follows from Lemma~\ref{lemma:easycomm}.
\end{proof}

	Since $A^* \in \Psibf^0$ when $A \in \Psibf^0$, Lemma \ref{lemma:mappingprop} yields the following corollary:
\begin{corollary}\label{cor:dualmappingprop}
	Let $A \in \Psibf^0$ and $u \in \Hone^*$. Then
	$$ \|Au\|_{\Hone^*}  \le C \|u\|_{\Hone^*} $$
	and $C$ depends only on the seminorms of the symbol of $A$.
\end{corollary}

\section{Sobolev spaces}
First, we recall the definition of b-Sobolev spaces, adapted to our
uniform context.
\begin{definition}
Let $s\geq 0$.  A distribution $u$ is in $\Hb^s(X)$ if for all $A \in
\Psibf^s(X),$ $A u \in L^2.$

The negative order spaces are defined dually: for $s>0,$ $u \in \Hb^{-s}(X)$
if $u \in \Psibf^s(X) (L^2)$.
\end{definition}
\begin{remark}\label{remark:Sob}
By elliptic estimates of the sort discussed below, we would obtain an
equivalent scale of Sobolev spaces by using the algebra $\Psib(X)$
rather than $\Psibf(X)$: the distinction is irrelevant here.  In
particular, we could use constant coefficient test operators, so that for $m \in 
\NN$,
$$
u \in \Hb^m(X) \Longleftrightarrow (r\pa_r)^i \pa_t^j \pa_\varphi^k u
\in L^2(X),\quad i+j+k\leq m.
$$
We nonetheless keep the $\F$ in the notation to remind the reader that these are
in almost every case the test operators under consideration (and the
distinction is more important in the $\Hone$-based spaces defined below).

Note that over compact sets, the distinctions between uniform and
ordinary b-Sobolev spaces
are also moot, and $\Hb^m\cap \mathcal{E}'=H_b^m \cap \mathcal{E}'$.
\end{remark}

We additionally define b-Sobolev regularity with background spaces given by
$\Hone$ or $\Hmone:$

\begin{definition}
Let $s\geq 0$.  A distribution $u$ is in $\HbHone^s(X)$ if for all $A \in
\Psibf^s(X),$ $A u \in \Hone.$ Likewise $u \in \HbHonestar^s(X)$ if for all $A \in
\Psibf^s(X),$ $A u \in \Hmone.$

The negative order spaces are defined dually: for $s>0,$ $u \in \HbHone^{-s}(X)$
if $u \in \Psibf^s(X) (\Hone)$ and likewise
for $\HbHonestar^s(X)$.
\end{definition}
  
As weighted b-Sobolev spaces also arise frequently, we introduce the
double-index notation
\begin{equation}\label{doubleindex}
\Hb^{s,l}(X)=r^l \Hb^s(X);
\end{equation}
thus $\Hb^{s,0}=\Hb^s,$ and we will continue to use the latter
notation when convenient.  When there is no possibility of confusion
(in particular, in the solvability argument in Section~\ref{sec:causal}
  below) we will write the norm
on $\Hb^{s,l}$ as $\norm{\bullet}_{s,l}.$
  
We will partially\footnote{We are omitting the ``bar'' used for spaces
  of restrictions, as this seems to be the more common usage in other
  literature.}  adopt the notation of \cite[Section B.2]{Ho:07}: let
$\Omega\subset X$ be an open set obtained as $\Omega=\beta^{-1}\Omega_0$ where
$\beta$ is the blowdown map and $\Omega_0 \subset \RR^3$ is an open
set with Lipschitz boundary.

For $\mathcal{Z}$ any of the Hilbert spaces
of distributions on $X$ introduced thus far, let
$$
\dot{\mathcal{Z}}(\overline{\Omega}),\ {\mathcal{Z}}(\Omega)
$$
denote respectively the subspace of $\mathcal{Z}$ consisting of
elements that are supported
on $\Omegabar\subset X,$ and the quotient space of $\mathcal{Z}$ consisting of
restrictions from $X$ to $\Omega$ of elements of $\mathcal{Z}.$

Note that since $X$ is a manifold with boundary, the open set $\Omega\subset X$ may contain points in $\pa
X=\{r=0\}$ (but as soon as it contains one point in $(r=0, t=t_0,
\varphi=\varphi_0)$ it must contain all points $(r=0, t=t_0,
\varphi\in S^1),$ as it is the preimage of a set in $\RR^3$).
The distinctions about boundary behavior that we are drawing among Sobolev spaces are
at other parts of the boundary, not at $r=0$.
  
  \begin{lemma}\label{lemma:spaces1}
    The following continuous inclusions of subspaces of
    $\mathcal{C}_c^\infty(X^\circ)'$ hold over any $\Omega, \Omegabar\subset
    \RR\times X$ as above, provided the function $r$ is bounded on $\Omega$:
    $$
\begin{aligned}
   \dot{H}^{1,1}_b (\Omegabar)&\subset \dot{\Hone}(\Omegabar)\subset \dot{H}^1_b(\Omegabar),\\
       H^{-1}_b (\Omega)&\subset \Hone^*(\Omega) \subset  H^{-1,-1}_b(\Omega)
\end{aligned}
    $$
  \end{lemma}
    \begin{proof}
      We write the squared norm on $\Hone$ as
      $$
 \|r^{-1} (rD_r )u\|_{L^2(X)}^2 + \|r^{-1}\fib u\|_{L^2(X)}^2 + \|D_tu\|_{L^2(X)}^2 + \| D_\varphi u\|_{L^2(X)}^2
          + \| u\|_{L^2(X)}^2;
          $$
      since the vector fields $r \pa_r$ and $\fib$ are in
      $\mathcal{V}_b$ this is certainly dominated by a multiple of the $r
      \Hb^1$ squared norm, given by
      $$
 \|r^{-1} (rD_r )u\|_{L^2(X)}^2 + 
 \|r^{-1} D_tu\|_{L^2(X)}^2 + \| r^{-1} D_\varphi u\|_{L^2(X)}^2
          + \| r^{-1} u\|_{L^2(X)}^2;
          $$
          (Here we have used boundedness of $r$ on $\Omega.$)  Thus  $r
          \dot{H}^1_b (\Omegabar)\subset \dot{\Hone}(\Omegabar).$  Likewise, the squared $\Hone$
          norm dominates
          $$
 \| (rD_r )u\|_{L^2(X)}^2 + \|D_tu\|_{L^2(X)}^2 + \| D_\varphi u\|_{L^2(X)}^2
          + \| u\|_{L^2(X)}^2,
          $$
which is the $H^1_b$ norm, hence $\dot{\Hone}(\Omegabar)\subset \dot{H}^1_b(\Omegabar).$  The
remaining inclusions follow by duality, using the usual duality
between spaces of extendible and supported distributions as discussed
in \cite[Section B.2]{Ho:07} (and for which the extensions to the
hypothesis of merely Lipschitz boundary can be found in \cite[Theorems
3.29, 3.30]{Mc:00}).
      \end{proof}

      \begin{corollary}\label{cor:uniformbounded}
Operators in $\Psib^s(X)$ are bounded from $\HbHone^{s-1}(X)$ to $L^2(X)$.
        \end{corollary}
        \begin{proof}
          Since $\Hone(X)\subset H_b^1(X)$, we additionally have
          $$
\HbHone^{s-1}(X) \subset H_b^s(X),
$$
and the result follows from boundedness of the uniform calculus.
          \end{proof}
\section{Ellipticity, microsupport, and wavefront set}

Since our test operators are required to be fiber invariant,
we cannot hope to use them to distinguish microlocal behavior at
different points in a fiber.  Hence we will define, for $s \in \RR
\cup \{\infty\},$
$$
\WFbh^s u,\ \WFbhstar^s u\subset \Sbstar X/{\sim}
$$
as subsets of the quotient space of the sphere bundle where we identify fibers in the base
space $X$ at $r=0.$ (Recall that $\sim$ lifts from $\pa X$ to give a
fibration of $\Sbstar_{\pa X} X$.)
We equip this space with the quotient
topology.

As we will deal very frequently in the sequel with objects living on
the quotients of $\Sbstar X$, $\Tbstar X$, and other spaces by $\sim$,  we henceforth adopt the
$\fib$ subscript notation
$$
\Tbstar X_\fib\equiv \Tbstar X/{\sim},\ \Sbstar X_\fib\equiv \Sbstar X/{\sim}
$$
to denote fiber quotients of these and other spaces.  Our previous
notation for $\Tbdotstar X$, matching that used in \cite{MeVaWu:08}, is admittedly slightly inconsistent with
this one, as it also involves a fiber quotient but has no subscript.
Note, in any event, that
$$
\Tbdotstar X \subset \Tbstar X_\fib,\ \Sbdotstar X \subset \Sbstar X_\fib.
$$

We begin with the usual definitions of ellipticity and microsupport, transported to the
quotient space $\Tbstar X_\fib,$ and with uniformity in the noncompact
fibers and fiber invariance built in via the uniform estimates on our symbols.


 In the following definition we use the quantization $\Op$ defined in
 Appendix~\ref{appendix:bcalculus} which is, by definition, a
 surjective map from an appropriate space of total symbols to the
 operator calculus (rather than requiring the manual addition of the
 residual operators, as we must do when using $\tOpb$).  The wavefront
 sets are defined as subsets of the fiber-quotiented cosphere bundle,
 but as usual we may equally well view it as a positive conic subset
 of the relevant cotangent bundle.
\begin{definition}
An operator $A=\Opb(a) \in \Psibf^s(X)$ is elliptic at $q \in \Sbstar X_\fib$ if
there exists $c \in S_{\fib}^{-s}$ such that  $ca=1$ in a
positive conic neighborhood of $q$ in $\Tbstar X_\fib$.   Let
$\liptic A$ denote the set of points at which $A$ is elliptic.

A point $q\in \Sbstar X_\fib$ is in the complement of the microsupport of $A=\Opb(a)$ (denoted
$\WFbf'(A)$) if there exists a positive conic neighborhood $U$ of
$q$ in $\Tbstar X_\fib$ on which for
all $N \in \NN$, and multiindex $\alpha$,
\begin{equation}\label{residualest}
\abs{\pa^\alpha a} \leq C_{N,\alpha} \ang{\tau,\xi,\eta}^{-N}.
\end{equation}

We use the same definition of microsupport for elements of
$\Psib(X)$, requiring the estimate \eqref{residualest} to hold on the
$\sim$-equivalence class of the neighborhood $U$.
\end{definition}
  Note then that ellipticity of an element of $\Psibf^0(X)$ at an
equivalence class of points over $r=0$ given by $\{(t+\C s, \varphi_0
+ s, \xib_0,\taub_0,\etab_0): s \in \RR\}$ implies uniform lower bounds on
the symbol on a set of the form
\begin{multline*}
\{(t+\C s, \varphi
+ s, \xib,\taub,\etab): \abs{\varphi-\varphi_0}<\ep, \\
\big\lvert(\xib,\taub,\etab)/\abs{(\xib,\taub,\etab)}-(\xib_0,\taub_0,\etab_0)/\abs{(\xib_0,\taub_0,\etab_0)}\big\rvert<\ep,\
\abs{(\xib,\taub,\etab)}>\ep^{-1}\}.
\end{multline*}

As usual, we have
$$
\WFbf' A =\emptyset \Longleftrightarrow A \in \Psibf^{-\infty}(X).
$$

In order to regularize our commutator arguments below, we will work with \emph{families}
of fiber-invariant operators which are uniformly bounded in some space of
uniform b-pseudodifferential operators. Thus we also require a notion
of microsupport in this context.  In the following definition
  and subsequent discussion, we use
the notion of bounded operator families in $\Psibf^k(X)$: this means
parametrized families of operators that are quantizations of symbols with all seminorms
uniformly bounded in the parameter.
\begin{definition}
  Let $\mathcal{B} \subset\Psibf^k$ be a bounded family of fiber-invariant
  operators indexed by $\ind$ in that there exists an indexing set
  $\mathcal{I}$ such that
  $\mathcal{B} = \{ B_\ind : \ind \in \mathcal{I}$\}. We say
  $q \in \Sbstar X_\fib$ is in the complement of the microsupport of
  the family (i.e., $q \notin \WFbf'\mathcal{B}$) if there exists a
  positive conic neighborhood of $q$ in $\Tbstar X_\fib$ on which
  for all $N \in \mathcal{N}$ and any multiindex $\alpha$, there
  exists $C_{N,\alpha}$ independent of $\ind$ such that
		\[ |\partial^\alpha b_\ind | \le C_{N,\alpha}\ang{\taub, \xib, \etab}^{-N} \]
	where $B_\ind = \Op_b(b_\ind)$.
\end{definition}
For a family $\mathcal{B}$,
$\WFbf'\mathcal{B} =\emptyset$
iff for all $B\in \mathcal{B},$ $B \in \Psibf^{-\infty}(X)$ \emph{with
  uniform estimates} on seminorms (which we will abbreviate in what follows as
``lying uniformly in $\Psibf^{-\infty}$'').

The following standard result can be read off from the locality in
$(x,\xi)$ of the formula for the total symbol of the composition, just
as in the usual calculus.
\begin{proposition}
$\WFbf'AB \subset \WFbf'A \cap \WFbf'B$ and the same inclusion holds for
operator families.
  \end{proposition}

There do exist elliptic operators in our calculus:
  \begin{lemma}
Given $p \in \Tbstar_{\pa X} X_\fib,$ $s \in \RR,$ and an open conic neighborhood
$U$ of $p,$ there exists $A \in \Psibf^s(X)$ that is elliptic at $p$
with $\WFbf' A\subset U.$
    \end{lemma}
    \begin{proof}
Letting $\chi$ be a cutoff function supported near the origin, we use
the notation $\alpha \equiv (\xib,\taub,\etab)$ and $\hat
\alpha=\alpha/\abs{\alpha}$, and set
$$
a= \chi(t-\C\varphi-\C \varphi_0) \chi(r) \chi(\lvert \hat \alpha-\hat
\alpha_0\rvert) \chi(\abs{\alpha}^{-1}).
$$
Then $A=\tOpb(a)$ satisfies our criteria, provided the support of
$\chi$ is sufficiently small.
\end{proof}

Global and microlocal elliptic parametrices exist in this calculus,
with the proof being the usual one (relying, as it does, only on the
properties of the symbol calculus---see, e.g., \cite[Proposition E.32]{DyZw:19}):
\begin{proposition}\label{prop:elliptic}
  Let $A \in \Psibf^\ell(X)$ and $B$ in $\Psibf^k(X)$ satisfy
  $$
\WFbf' A \subset \liptic B.
  $$
  Then there exists $Q,Q' \in \Psibf^{\ell-k}(X)$ such that
  $$
A=BQ+R=Q'B +R'
$$
with
$$
R,R'\in \Psibf^{-\infty}(X)
$$
and $\WFbf' Q\cup \WFbf' Q' \subset \WFbf' A$.

The same result holds if $A \in \Psib^\ell(X)$ and $B \in
\Psibf^k(X)$, now with $R, R'\in R,R'\in \Psib^{-\infty}(X)$.
\end{proposition}
Note in particular, taking $A=\Id$, we can invert a globally elliptic
operator modulo a residual operator.
  
As usual, elliptic elements of the calculus can be used to test for
Sobolev regularity:
  \begin{proposition}
Let $A\in \Psibf^s(X)$ be elliptic.  Then \begin{align*} u \in
                                            \Hb^s(X)  &\text{ iff }A u
\in L^2\\ u \in \HbHone^s(X) &\text{ iff }
A u \in \Hone\\ u \in \HbHonestar^s(X) &\text{ iff }A u \in \Hmone.\end{align*}
\end{proposition}
The proof follows directly from Proposition~\ref{prop:elliptic} and
from boundedness of nonpositive-order elements of the calculus on
$L^2$, $\Hone$, $\Hmone$.

\begin{definition}
  Let $u \in \CmI(X)$ and $q \in \Tbstar X_\fib.$
  We say $q \notin \WFbh^su$ (resp.\ $q \notin \WFbhstar^s u$, $q \notin \WFb^s u$)
        if there is an $A \in \Psibf^s,$ elliptic at $q,$ such that $Au \in
  \Hone$ (resp. $Au \in \Hmone$, $Au \in L^2$).

We say $q \notin \WFbh u$ (resp. $q \notin \WFbhstar u$, $q \notin \WFb u$)
        if there is an $A \in \Psibf^0,$ elliptic at $q,$ such that $Au \in
  \HbHone^\infty(X)$ (resp.\ $Au \in \HbHonestar^\infty$, $Au \in H_b^\infty$).
  We occasionally use the convention that $\WF_\bullet^\infty$ means the same thing
  as $\WF_\bullet$ with no superscript.
\end{definition}

Standard ellipticity arguments show that with the definitions above,
$$
\WF_\bullet^\infty u=\overline{\bigcup_{s \in \RR} \WF_\bullet^su },
$$
hence we will not need to deal separately with $s=\infty$ in our
results below.

We have an inclusion of wavefront sets based on the corresponding
inclusions of Sobolev spaces:
\begin{lemma}\label{lemma:WFinclusions}
      For $u \in \Hone,$ and $m \in \RR\cup \{+\infty\},$
\begin{equation}\label{WF1}
\WFb^{m+1} (u) \subset \WFbh^m(u)\subset \WFb^{m+1} (r^{-1} u).
\end{equation}
      For $u \in \Hmone,$ and $m \in \RR \cup \{+\infty\},$
\begin{equation}\label{WF2}
\WFb^m(ru) \subset\WFbhstar^{m+1}(u)\subset \WFb^m(u) 
\end{equation}
    \end{lemma}
    \begin{proof}
If a point $q \notin \WFbh^m(u),$ there exists $A \in \Psibf^m,$
elliptic at $q,$ such that $A u \in \Hone.$  Hence by
Lemma~\ref{lemma:spaces1}, $Au \in \Hb^1,$ and we obtain
$q \notin \WFb^{m+1}u,$ which proves the first inclusion in \eqref{WF1}.
If $q \notin \WFb^{m+1} (r^{-1} u)$ then there exists $A \in
\Psibf^{m+1}$ with $r^{-1} A u \in L^2,$ hence for $B \in \Psibf^m$
microsupported on the elliptic set of $A$ and elliptic at $q,$ we have
$B u \in r \Hb^1 \subset \Hone$ (again by Lemma~\ref{lemma:spaces1}) hence $q \notin \WFbh^m(u)$ and the
second inclusion follows.

The inclusions in \eqref{WF2} are proved similarly.
      \end{proof}

\section{Microlocal elliptic regularity}\label{sec:elliptic}

In this section we prove a microlocal elliptic regularity statement
for the weighted edge operator $P$. This says, quantitatively, that away from the
wavefront set of the inhomogeneity, the wavefront set of a solution to
$Pu = f$ is contained in the compressed b-cotangent bundle defined in
Section~\ref{sectionCompressedBundle}. 

Recall that $$P = -\partial_t^2 + r^{-2}(r\partial_r)^2 + r^{-2} (\C \pa_t + \pa_\varphi)^2 + \pert$$
with
$$ \pert = f_1 D_t + f_2 D_\varphi + f_3rD_r + f_4 $$
where $f_\bullet \in \mathcal{C}_u^\infty$.

\begin{lemma}\label{lemma:ubddtoQ}
  Take $\Gamma \subset U \subset \Sbstar X_\fib$ with $\Gamma $ closed and
  $U$ open. Let $\mathcal{B}\subset \Psibf^k$ be a family
  uniformly bounded in $\Psib^{k}$ with
  $\WFbf'(\mathcal{B}) \subset \Gamma $.  Let $Q \in \Psibf^k$ be 
  elliptic on $\Gamma$ with $\WFbf'(Q) \subset U$. 
  
  \begin{itemize}
  	\item[(i)] If $u \in \Hone$ and	$\WFbh^k(u) \cap U = \emptyset$ then there exists a
  	constant $C_1 >0$ such that
  	$$
  	\|Bu\|_\Hone \le C_1(\|u\|_\Hone + \|Qu\|_\Hone)
  	$$
  	for $B\in \mathcal{B}$.
  	
  	\item[(ii)] If $ u\in \Hone^* $ and $\WFbhstar^k(u) \cap U = \emptyset$ then there exists a constant $C_2 > 0$ such that
  		$$
  		\|Bu\|_{\Hone^*} \le C_2 (\|u\|_{\Hone^*} + \|Qu\|_{\Hone^*})
  		$$
  \end{itemize}
\end{lemma}

\begin{proof}
  We begin by proving statement $(i)$. Since $\Gamma$ is closed, it
  follows from Proposition~\ref{prop:elliptic} that
    we can find a microlocal parametrix $G$ for $Q$ so that
  $GQ = \Id + E $ where $E \in \Psibf^0(X)$,
  $\WFbf'(E) \cap \Gamma = \emptyset$. For each $B \in \mathcal{B}$,
  $BE \in \Psibf^{-\infty}(X)$ uniformly since
  $\WFbf'(B)\cap \WFbf'(E) = \emptyset$. Thus
	\begin{align*}
		\|Bu\|_\Hone &= \|B(GQ-E)u\|_\Hone \\
			&\le \|BGQu\|_\Hone + \|BEu\|_\Hone \\
			&\le C_1(\|u\|_\Hone + \|Qu\|_\Hone)
	\end{align*}
	where in the last line we bound $\|BGQu\|_\Hone$ using Lemma
        \ref{lemma:mappingprop} and the second using the fact
        that $BE \in \Psibf^{-\infty}$ uniformly.
	
	The proof of statement $(ii)$ is identical but we use Corollary \ref{cor:dualmappingprop} instead of Lemma \ref{lemma:mappingprop}.
\end{proof}

The following lemma relates our $\Hone$-based wavefront set to $L^2$
bounds.   Note that the test operator $A$ is merely uniform, but the
operator $G$ estimating it is fiber-invariant.

\begin{lemma}\label{lemma:wavefrontHvsL2}
	Take $\Gamma \subset U \subset \Sbstar X_\fib$ with $\Gamma$ closed
        and $U$ open.  Let $u \in \Hone$ and let $A \in
        \Psib^{m+1}(X)$ for some $m \in \mathbb{R}$. Assume $\WFbf'(A)
        \subset \Gamma$ and $\WFbh^m u \cap U = \emptyset$. Then for
        any $G \in \Psibf^m$ with $\WFbf' G \subset U$ and $G$ elliptic
        on $\Gamma$,
there exist $C >0$ and
        $R \in \Psib^{-\infty}$ such that

	$$\|Au\|_{L^2} \le C \|Gu\|_\Hone+\|Ru \|_{L^2}.$$
\end{lemma}

	\begin{proof}
	 Fixing $G
                \in \Psibf^m(X)$ elliptic on $\Gamma$ and microsupported
                in $U$, the elliptic parametrix construction
                (Proposition~\ref{prop:elliptic})
                  gives the
                factorization
                $$
A=TG+R
                $$
                with $T \in \Psib^1$, $R \in \Psib^{-\infty}$.  We
                then estimate, using Corollary~\ref{cor:uniformbounded},
                $$
\|T G u \|\leq C \|Gu\|_{\Hone}
$$
and the desired estimate follows.

	\end{proof}

In our proof of estimates on the Dirichlet form below, we will need to control many terms which share similar structures. To streamline the argument, we provide the following lemma which will allow us to immediately bound the necessary terms. 

\begin{lemma}\label{lemma:movingderivatives}
  Take $\Gamma \subset U \subset \Sbstar X_\fib $ with $\Gamma $ closed and
  $U$ open. Let $\mathcal{A}, \mathcal{B}$ be families of operators
  uniformly bounded in $\Psibf^{k}$ and $\Psibf^{k'}$,
  respectively. Assume
  $\WFbf'(\mathcal{A}), \WFbf'(\mathcal{B}) \subset \Gamma.$ Let
  $E, F$ be weighted differential operators of respective orders $m,m' \in \{0, 1\}$ lying in
  $\Diffbf+ \CIf r^{-1}\fib + \CIf D_r$.
Set $$ n = \frac{m+m'+k +k'-2}{2}.$$ Assume
  $u \in \Hone$ and $\WFbh^n u \cap U = \emptyset$. Then for any
  $Q \in \Psibf^n$ elliptic on $\Gamma$ such that $\WFbf'(Q) \subset U$,
  there exists a constant $C >0$ such that
		$$ |\la EAu, FBu \ra | \le C \big(  \|u\|_\Hone^2 + \|Qu\|_\Hone^2 \big) ,$$
	for $A \in \mathcal{A}, B \in \mathcal{B}$.
\end{lemma}

\begin{proof}
	Set $$\ell = \frac{-m + m' -k + k'}{2}.$$ Let $\Lambdaplus \in \Psibf^\ell(X)$ be elliptic with parametrix $\Lambdaminus \in \Psibf^{-\ell}(X)$ satisfying
		$$ \Lambdaminus \Lambdaplus = \Id + R $$
	with $R \in \Psibf^{-\infty}(X)$.
	
	Then using Lemma~\ref{lemma:easycomm} to deal with commutators
        of $r^{-1}$ and $D_r$ with elements of $\Psibf(X)$ 
        we calculate
	\begin{align*}
		 |\la EAu, FBu \ra |  &= | \la (\Lambdaminus \Lambdaplus - R) E A u, F B u \ra | \\	
		 	&\le | \la \Lambdaplus E Au, \Lambdaminus^*F B u \ra | + | \la R EA u, FBu \ra| \\
		 	&= \sum|\la \tilde{E}_j \tilde{A}_j u, \tilde{F}_j \tilde{B}_j u \ra | +  | \la R EA u, FBu \ra|
	\end{align*}
	where the sum is a finite one with
        $\tilde{E}_j$, $\tilde{F}_j$ having the same form and order as
        $E$ resp.\ $F$, and where $\tilde{A}$, $\tilde{B}$
        are members of families of operators $\tilde{\mathcal{A}}$
        (resp. $\tilde{\mathcal{B}}$) which are
        uniformly bounded in $\Psibf^{k+\ell}$, respectively
        $\Psibf^{k'-\ell}$.
	
	By Cauchy--Schwarz,
	\begin{equation} \label{eq:cauchyschwarz}
			|\la \tilde{E}_j \tilde{A}_j u, \tilde{F}_j \tilde{B}_j u \ra | \lesssim \|\tilde{E}_j \tilde{A}_ju\|_{L^2}^2 + \| \tilde{F}_j \tilde{B}_j u \|^2_{L^2}.
	\end{equation}
	For the first term on the RHS of 
        \eqref{eq:cauchyschwarz}, if $m=1$ we see $ n = k+\ell$ and we use Lemma \ref{lemma:ubddtoQ} to find
		$$ \|\tilde{E}_j \tilde{A}_j u \|_{L^2}^2 \lesssim \|\tilde{A}_j u \|_\Hone^2 \lesssim \|u\|_\Hone^2 + \|Qu\|_\Hone^2. $$
	If $m=0$ then $n = k+\ell -1$ and we use Lemma
        \ref{lemma:wavefrontHvsL2} along with Lemma
        \ref{lemma:ubddtoQ} to likewise obtain
		$$ \| \tilde{E}_j \tilde{A}_j u\|_{L^2}^2 \lesssim \|u\|^2_\Hone + \|Qu\|_\Hone^2. $$
	The second term on the right hand side of \eqref{eq:cauchyschwarz} is handled analogously.
\end{proof}

Next we establish an estimate on the Dirichlet form. Note that this lemma is an analogue of Lemma 8.8 from \cite{MeVaWu:08}. 

\begin{lemma}\label{lemma:ellest}
Let $A_\ind  \in \Psibf^{s-1}$ be a family of
        operators indexed by $\ind \in I$ that are uniformly bounded in $\Psib^s$ with
        $\WFbf'(A_\ind ) \subset \Gamma \subset U \subset \Sbstar X_\fib$ for $\Gamma$
        closed and $U$ open. Assume $u \in \Hone$ satisfies $\WF_{b,\Hone}^{s-\frac{1}{2}}(u) \cap U = \emptyset$.  There exist
       $Q \in \Psibf^{s-\frac{1}{2}},$ $\tQ \in \Psibf^{s+\frac 12}$
       elliptic on $\Gamma$ with $\WFbf'(Q)\cup \WFbf'(\tQ) \subset U$ and a constant $C >0$ such that
       \begin{equation}\label{dirichletformbnd}
         \begin{aligned}
		\Big| \|D_t A_\ind  u \|_{L^2} &-\|D_r A_\ind
                u\|_{L^2} -  \|r^{-1}\F A_\ind
                u\|_{L^2} \Big| \\ &\le C\left( \|u\|^2_{\Hone} +
                  \|Qu\|^2_{\Hone} + \|P u\|_{\Hmone}^2 +
                    \|\tQ P u\|_{\Hmone}^2 \right)\end{aligned}
	\end{equation}
	for all $\ind \in I$.
	\end{lemma}

\begin{proof}   Our assumptions on the wavefront set of $u$ imply that  for each $\ind$, we have $A_\ind  u \in \Hone$ so that $P A_\ind  u \in \Hone^*$. It follows that
	\begin{align*} 
          \|D_t A_\ind  u \|_{L^2} -& \|D_r A_\ind  u\|_{L^2}
                                      -\|r^{-1}\F A_\ind u\|_{L^2} = \langle P A_\ind  u, A_\ind u \rangle - \langle \pert A_\ind u, A_\ind u \rangle \\
                                    &= \langle [P, A_\ind ] u, A_\ind  u \rangle + \langle A_\ind P u , A_\ind u \rangle - \langle \pert A_\ind u, A_\ind u \rangle.
	\end{align*}
 Thus it suffices to show
\begin{equation}\label{commterm}
	|\langle [P, A_\ind  ] u, A_\ind  u \rangle - \langle \pert A_\ind u, A_\ind u \rangle | \le C \left( \|u\|_\Hone^2 + \|Qu\|_\Hone^2 \right)
\end{equation}
and
\begin{equation}\label{puterm}
	|\langle A_\ind P u , A_\ind u \rangle|  \le C \left( \|u\|_\Hone^2 + \|Qu\|_\Hone^2 +  \| P u\|_{\Hone^*}^2 + \|\tilde{Q}Pu\|_{\Hone^*}^2 \right).
\end{equation}

First we establish \ref{puterm}. Using the notation from the proof of Lemma \ref{lemma:movingderivatives} we have
	\begin{align*}
		|\la A_\ind P u, A_\ind u \ra | \le | \la \Lambda_{\frac{1}{2}}A_\ind P u, \Lambdaneghalf^*A_\ind u \ra| + | \la A_\ind Pu, R^*A_\ind u \ra| 
	\end{align*}
where $R \in \Psibf^{-\infty}$. The second term is readily bounded as in \eqref{dirichletformbnd}.  For the first term we use Cauchy--Schwarz and Lemma \ref{lemma:ubddtoQ} to find
 	\begin{align*}
          \left| \la \Lambda_{\frac{1}{2}}A_\ind P u,
          \Lambdaneghalf^*A_\ind u \ra \right| &\lesssim \|
                                                 \Lambdahalf A_\ind P
                                                 u \|_{\Hone^*}^2 + \|
                                                 \Lambdaneghalf A_\ind
                                                 u\|_\Hone^2 \\ 
                                               &\lesssim \|P u\|_{\Hone^*}^2 + \|\tilde{Q}P u \|_{\Hone^*}^2 + \|u\|_\Hone^2 + \|Qu\|_\Hone^2,
 	\end{align*}
 which proves \ref{puterm}.

 As we turn our attention to the commutator term in \ref{commterm}, we
 are primarily interested in the order of the pseudodifferential
 operators that arise. Since we commute with $A_\ind \in \mathcal{A}$,
 these operators will be members of families indexed by $\mathcal{A}$. To
 concisely track the relevant information we use the notation $B_k$ to
 indicate a representative of the class of families of operators
 indexed by $\ind$ which are elements of $\Psibf^{k-1}(X)$ and are
 uniformly bounded in $\Psibf^{k}(X)$. The precise operator being
 specified may change at each appearance of $B_k$.

Now consider $[P,A_\ind ]$. As $D_r^* = D_r-\frac{i}{r}$, $D_r^*D_r =
- r^{-2}(r\pa_r)^2$, hence
	\[ [\Box_g,A_\ind ] = B_{s+1} + r^{-2}\F^2 B_{s-1} - [D_r^*D_r,A_\ind ].\] 
Recalling that $\pert$ is a first-order b-differential operator lying
in $\Psib^1(X)$, we find that $[\pert, A_\ind]$ is uniformly bounded in $\Psib^s(X)$.
Thus
	\[  [P,A_\ind] = B_{s+1} +B'_s+ r^{-2}\F^2 B_{s-1} -
          [D_r^*D_r,A_\ind ]  + D_r B_{s-1},\]
        where $B'_s \in \Psib^s(X)$ is uniformly bounded and all other
        families are in $\Psibf(X)$.
Thus
\begin{align*}
	\la [ P,&A_\ind ] u, A_\ind  u \ra - \la \pert A_\ind u, A_\ind u \ra \\
			 &= \la B_{s+1} u, A_\ind u \ra +  \la r^{-1}\fib B_{s-1} u, r^{-1}\fib A_\ind u \ra -  \la [D_r,A_\ind ]u, D_rA_\ind u \ra  \\
					&\quad  + \la D_ru,[D_r,A_\ind ^*]A_\ind  u \ra + \la D_r B_{s} u, A_\ind u \ra +\la B'_s u, A_\ind u \ra\\
			&= \la B_{s+1} u, A_\ind u \ra +  \la r^{-1}\fib B_{s-1} u, r^{-1} \fib A_\ind u \ra - \la B_s u, D_r A_\ind u \ra  \\ 
				&\quad + \la D_r B_{s-1} u, D_r A_\ind u \ra + \la D_r u, B_{2s}  u \ra + \la D_r u, D_r B_{2s-1} u \ra \\
				&\quad
                           A_\ind u \ra + \la D_r B_{s} u, A_\ind u
                           \ra+\la B'_s u, A_\ind u \ra.
\end{align*}


Each term but the $B'_s$ term is then controlled by $\|u\|_\Hone^2 +
\|Qu\|_\Hone^2$, by Lemma \ref{lemma:movingderivatives} since, using
the notation of Lemma \ref{lemma:movingderivatives}, in each term of
our expression for $\la [P, A_\ind] u, A_\ind u \ra$ above  we see $m
+ m' + k + k' \le 2s+1$. The $B'_s$ term is likewise controlled by
Corollary~\ref{cor:uniformbounded}.
This concludes the proof of \ref{commterm}.
      \end{proof}

We conclude the section with the proof of our microlocal elliptic
regularity statement.

	\begin{proposition}\label{prop:ellipticreg}
Let $u \in \HbHone^{-N}$ for some $N\in \RR$. Then
                        for any $m \in \RR \cup \{+\infty\},$
                        			\[ \WFbh^m
                                                  (u)\backslash
                                                  \WFbhstar^m(Pu)
                                                  \subset\Sbdotstar X.  \]
	\end{proposition}
Recall that $\Sbdotstar X$ is the fiber-quotient of the image
        of the edge-cotangent bundle inside the b-cotangent bundle,
        defined in Section~\ref{sectionCompressedBundle}.
	\begin{proof}
          Note that $\WFbh^{-N}(u) = \emptyset$ since $u \in \HbHone^{-N}$. We
          claim that
          $q \in \Sbstar(X) \setminus \dot{\pi}(\Sestar(X))$,
          $q \notin \WFbh^s(u)$ and
          $q \notin \WFbhstar^{s+\frac{1}{2}} (Pu)$ implies
          $q \notin \WFbh^{s+\frac{1}{2}}(u)$. The proposition then
          follows by induction.  To establish the claim, we first note
          for $q \in \Sbstar X_\fib \setminus \Sbdotstar X$ we have
			\[ \taub^2 < \xib^2 + (\etab + \C \taub)^2. \]

		Take $A \in \Psibf^{s+\frac{1}{2}}$ such that $A$ is
                elliptic at $q$ and $\WFbf'(A) \cap \WF_{b,\Hone}^s(u)
                = \emptyset.$ Let $a$ be the symbol of $A$ and define
                $A_\ind$ as the quantization of $\big(1 + \ind(\taub^2 +\xib^2 + (\etab + \C \taub)^2)\big)^{-1/2}a$. Then $A_\ind \in \Psibf^{s- \frac{1}{2}}$,  $A_\ind$ is uniformly bounded in $\Psibf^{s+\frac{1}{2}}$, and $A_\ind \to A$ as $\ind \to 0$.
			
		Define $B\in \Psibf^1(X)$ as the quantization of
                $\left( \chi_q\left( \frac{1}{2\delta^2}(\xib^2
                    +(\etab + \C \taub)^2) - \taub^2
                  \right)\right)^{\frac{1}{2}}$ where $\chi_q$ is a
                cutoff supported near $q$ with $\chi_q$ and $\delta$
                chosen so that $B$ is elliptic near $q$. Furthermore
                we assume
                $\WFb'(\hat{\chi}_q -1) \cap \WFb'(A_\ind) = \emptyset$
                where $\hat{\chi}_q = \Op(\chi_q)$. Then
		$$B^*B = \left( \frac{1}{2\delta^2}  [(rD_r)^2 + \F^2] - D_t^2  \right) [1 + (\hat{\chi}_q-1)] + G$$
		for some $G \in \Psibf^1(X)$. Our assumption on the wavefront set of $\hat{\chi}_q-1$ and $A$ then give
		$$ B^*B A_\ind =  \left( \frac{1}{2\delta^2}  [(rD_r)^2 + \F^2] - D_t^2  \right)A_\ind + GA_\ind + E_\ind $$
		where $E_\ind \in \Psibf^{-\infty}$.

 We may further assume $A_\ind $ is supported in $r < \delta$ so that
		\begin{align*}
			\|BA_\ind  u \|_{L^2}^2 - & \int_X G A_\ind  u \overline{A_\ind  u}\\
			 &= \frac{1}{2\delta^2} \big( \|rD_rA_\ind  u\|_{L^2}^2 + \|\F A_\ind u\|_{L^2}^2\big) - \|D_t A_\ind u \|_{L^2}^2 \\
				&\le \frac{1}{2} \big( \|D_rA_\ind  u\|_{L^2}^2 + \|r^{-1}\F A_\ind u\|_{L^2}^2  \big) - \|D_t A_\ind u \|_{L^2}^2. 
		\end{align*}		
Since $\int_X G A_\ind  u \overline{A_\ind  u}$
  is uniformly bounded by the inductive hypothesis,
by Lemma \ref{lemma:ellest} we see
	\[ \frac12 \big( \|D_rA_\ind  u\|_{L^2}^2 + \|r^{-1}\F A_\ind u\|_{L^2}^2  \big)  + \|BA_\ind u\|_{L^2}^2 \]
is uniformly bounded as $ \ind \to 0$. It follows that $ \|A_\ind u\|_\Hone$ is uniformly bounded as $\ind \to 0$. Thus there is a subsequence of $A_\ind u$ which converges weakly in $\Hone$. Since $A_\ind \to A$, we see the weak limit is $Au \in \Hone$ so that $q \notin \WFbh^{s+\frac{1}{2}}$, as desired.

	\end{proof}

We also state a less precise version of this result, involving only
b-regularity:
\begin{corollary}\label{cor:b-elliptic}
Let $u \in \HbHone^{-N}(X)$ for some $N \in \RR$. Then
                        for any $m \in \RR \cup \{+\infty\},$
                        			\[ \WFb^{m+1}
                                                  (u)\backslash
                                                  \WFb^{m-1}(Pu)
                                                  \subset \Sbdotstar X.  \]

  \end{corollary}
  \begin{proof}
This follows directly from Proposition~\ref{prop:ellipticreg} and the
wavefront set inclusions in Lemma~\ref{lemma:WFinclusions}.
    \end{proof}
        \section{Law of reflection}\label{sec:propagation}
        
        In this section we prove our propagation of singularities
        result using a positive commutator argument.
We begin with a lemma
        giving the explicit form of the relevant commutator.

        \begin{lemma}\label{lemma:maincommutator}
          Let $A \in \Psibf^m(X)$ with real
          principal symbol.  Then
\begin{equation}\label{maincommutator}
i[\Box_g,A^*A] = D_r^* L_r D_r +(r^{-1} \fib) L_\fib (r^{-1} \fib) + D_r^* L' +
L'' D_r + L_0.
\end{equation}
where
\begin{itemize}
\item
  $\displaystyle L_r,L_\fib \in \Psibf^{2m-1},\ \sigmab{2m-1}(L_{\bullet})= 4 a \pa_{\xib} a;
$
\item
  $\displaystyle L',L'' \in \Psibf^{2m}(X),\
  \sigmab{2m}(L')=\sigmab{2m}(L'') = 2 a \pa_{r} a$
  \item $\displaystyle L_0 \in \Psibf^{2m+1}(X),\ \sigmab{2m+1}(L_0)
    =4 \tau a \pa_ta.$
\end{itemize}
\end{lemma}
\begin{proof}
  Writing
  $$
  \Box_g=D_t^2-D_r^* D_r -r^{-2}(\fib)^2,
  $$
  we first note simply that $D_t^2 \in \Psibf^2(X),$ so that by
  ordinary properties of the calculus,
  $$
i[D_t^2, A^*A]\in \Psibf^{2m+1}(X)
$$
with symbol
$$
2\tau \pa_t(a^2),
$$
yielding the term $L_0.$

Writing
$$
[D_r^*D_r, A^*A]= D_r^*[D_r, A^*A] - [D_r, A^*A]^*D_r.
$$
By Lemma~\ref{lemma:easycomm},
$$
[D_r, A^*A] = E + FD_r
$$
where $\sigmab{2m}(E) = -i\pa_r(a^2)$ and $\sigmab{2m-1}(F)=-i\pa_{\xib}(a^2).$
Since these are purely imaginary,
$$
-[D_r, A^*A]^* = E' + D_r^* F'
$$
where $E'$ and $F'$ have the same symbols as $E$ and $F$
respectively.  Thus
$$
[D_r^*D_r, A^*A]=D_r^* L_r D_r + D_r^* L'+L''D_r.
$$
Here $\sigmab{2m}(L_r) = -2i \pa_r(a^2),$.

Finally, $\fib$ commutes with $A^*A$ by fiber-invariance, and by Lemma~\ref{lemma:easycomm}
$$
[r^{-2}, A^*A]=r^{-1}[r^{-1}, A^*A] + [r^{-1}, A^*A] r^{-1}
$$
equals $r^{-1}L_\fib r^{-1}$ where $L_\fib \in  \Psibf^{s-1}$ has principal symbol $-2i\pa_{\xib}(a^2).$ Note that we have used
our freedom to write $r^{-1}W=W'r^{-1}$ where $W,W'$ have the same
principal symbols.
  \end{proof}

We now state our main propagation theorem in precise form.
        \begin{theorem}\label{theorem:reflection}
          Let $u\in \HbHone^{-\infty}(X)$ be a solution to $P u=f.$  Let
          $\varrho=(\varphi_0, \tau_0) \in \Sbdotstar_{\pa X} X_\fib$ and let
          $U$ be an open neighborhood of $\varrho$ in
          $\Sbstar  X_\fib.$  Then (for either choice of $\pm$ below),
          $$
          \begin{array}{l}
          \WFbhu^s u\cap  U \cap \{\pm\xib<0\} =\emptyset, \\
            \WFbhstaru^{s+1} f \cap U =\emptyset\end{array}
\Longrightarrow \varrho \notin \WFbh^s u.
          $$
          \end{theorem}
          \begin{remark} \label{rem:big}\mbox{}\\
\begin{enumerate}\item
            Recall that since $\varrho \in \Sbdotstar_{\pa X} X_\fib,$ absence of
        $\varrho$ from the wavefront set is a \emph{global,
          fiber invariant} statement (as are the hypotheses, with
        uniformity along the fibers built in).
\item \label{propagationresult}
  Say $f=0.$  By elliptic regularity there is a priori no
  wavefront set of $u$ in $\xib\neq 0$ over $\pa X,$ so the hypothesis of
  the theorem can be viewed as dealing only with the interior of $X,$ where
  the points with $\pm\xib<0$ are those that are ``incoming'' toward the
  string or ``outgoing'' from it.  In particular, say we assume that there is no wavefront set
  of $u$ (uniformly) near all rays with $\taubhat=\tau_0$ striking a single fiber
  $\varphi-t/\C=\varphi_0,$ i.e., near the set  $\fcal_{I, \varphi_0,
    \tau_0};$ to further clarify signs, let us take $\tau_0=+1$ so
  that $\xib>0$ on $\fcal_{I,\varphi_0, \tau_0}.$  The uniformity
  assumption means there exist $r_0>0,$ $\delta>0$ such that (using
  the homogeneous ``hat'' coordinates of \eqref{hatcoordinates})
$$
\{r\in (0, r_0),\ \abs{\varphi-t/\C-\varphi_0}<\delta,\ \mp
  \xibhat\in (0,\delta),\ \sgn \tau=\tau_0\} \cap \WF u=\emptyset,
  $$
  with uniform estimate (i.e.\ by testing by an element of $\Psibf(X)$).
In particular, then, for $r<\min(r_0,\delta/2),$ and
$\abs{\varphi-t/\C-\varphi_0}<\delta$,  $\sgn \tau=\tau_0$, the
  points with $  \xibhat \in (0,\delta)$ are not in the wavefront set of $u$; additionally $\{\xibhat>\delta\}\cap \WF u=\emptyset$ since this set is disjoint from $\Sigma.$  The set
  $r=0, \xibhat<0$ is also disjoint from the wavefront set by elliptic
  regularity.  Hence we have shown that the hypotheses of
  Theorem~\ref{theorem:reflection} are satisfied at $(\varphi_0,
  \tau_0).$

  Thus the theorem can be interpreted as saying that uniform
  regularity along $\fcal_{I,\varrho}$ yields regularity at $\varrho$
  itself, and hence along $\fcal_{O,\varrho}$ as well, by closedness of
  wavefront set.  Hence this is the sought-after propagation of
  singularities into and out of the fiber.

  The theorem includes the converse statement as well: outgoing
  regularity yields incoming regularity, by backward-in-time
  propagation.  Note, though, that if $\Box_g u=f,$ then setting
  $\tilde{u}=u(-t,x)$ and $\tilde f=f(-t,x)$ implies
  $$
\widetilde{\Box_g}\tilde u=\tilde f,
$$
where $\widetilde{\Box_g}$ is the the d'Alembertian for the cosmic
string with $\C$ replaced by $-\C.$  Hence it will suffice to show
that incoming regularity implies outgoing regularity (for every
value of $\C$).  It also suffices to fix one sign of $\tau_0:$ since
$\Box_g$ has real coefficients, applying a propagation theorem valid for
$\tau_0>0$ to the equation $\overline{P} \overline{u}=\overline{f}$ proves the
corresponding result for $\tau_0<0.$
\item
Consider a solution $u$ to $P u=0$ that has a single
bicharacteristic in the
wavefront set arriving at $\pa X$ at a point $\varrho \in \Sbstar_{\pa X}X.$  Applying the
theorem in all fibers \emph{except} that containing $\varrho$
tells us that a single point in $\WFbh u$ striking the cosmic string may at most produce
$\WFbh u$ leaving the string \emph{everywhere} along the fiber of
$\varrho,$ in the same $\tau$-component of the characteristic set;
this is to say that $\varphi-t/\C$ and $\tau$ are conserved in the
interaction, but $t$ is (apparently) not.
\item
  When one proves microlocal propagation of regularity, the
contrapositive is usually interpreted as yielding
propagation of wavefront set along bicharacteristics, perhaps
of a generalized sort.  Here, the generalization would have to be
quite broad to permit such a statement: a singularity arriving at
$r=0$ along some ray in $\fcal_{I,\varphi_0, \tau_0}$ must result in a failure of uniform wavefront set
estimates along all rays in $\fcal_{O,\varphi_0, \tau_0}$. It is the \emph{uniformity} that is the difficulty
here: rather than produce wavefront set along some particular outgoing ray, the
effect might be merely to produce nonuniformity of estimates
along this family of rays as $\abs{t} \to \infty.$  An appropriately
defined notion of wavefront set at timelike infinity might allow us to
recover a more traditional propagation statement, but we will not pursue it here.
\item
Our hypotheses on the inhomogeneity $f$ are phrased in terms of
regularity with respect to $\Hmone,$ the dual space to $\Hone,$ which
away from $r=0$ agrees with $H^{-1};$ thus the hypotheses involve one
less derivative on $f$ than the concluded regularity on $u,$ as befits
a second-order hyperbolic equation.
\end{enumerate}
\end{remark}

\begin{proof}
The strategy of proof is descended from the work of
Melrose--Sj\"ostrand \cite{MeSj:78}, \cite{MeSj:82}
  and is more directly inspired by the presentations of Vasy
  \cite{Va:04} and then Melrose--Vasy--Wunsch \cite{MeVaWu:08}.
  
As noted in the remarks above, it suffices to consider the case where
our wavefront set assumption is in $\xibhat>0$ and we also have $\tau_0>0.$
        
        We construct a test operator $A$ as follows. Define
        $$ 
\omega=r^2 + (\varphi-t/\C-\varphi_0)^2
$$
(where as always, we view $\varphi-t/\C$ and $\varphi_0 \in\RR/ 2\pi\ZZ$ as
equivalence classes
with the distance squared being the shortest of the possible values)
and set
$$
\phi=-\hat\xib+\frac{1}{\beta^2 \delta}\omega
$$
where $\beta$ and $\delta$ are parameters to be set later. Let $\chi_0$ be smooth, supported in $[0,\infty),$ and with
$\chi_0(s)=e^{-1/s}$ for $s>0,$ hence $\chi'_0(s) = s^{-2} \chi_0(s).$
Let $\chi_1$ be supported on $[0,\infty)$ and equal to $1$ on
$[1,\infty)$ and with $\chi_1'\geq 0$ and supported in $(0,1).$  Let
$\chi_2 \in \mathcal{C}_c^\infty(\RR)$ be supported in $[-2c_1,2c_1]$
and equal to $1$ on $[-c_1,c_1]$ for some $c_1 > 0$.  Let $\chi_3=\chi_1.$ Our test operator $A$ will be defined to have principal symbol
$$
a= \chi_0\big(1-\frac \phi\delta\big)\chi_1\big(-\frac
{\xibhat}\delta+1\big) \chi_2(\xibhat^2+(\C
\taubhat+\etabhat)^2)\chi_3(\sgn (\tau_0) \taubhat).
$$
This symbol is fiber invariant, as it depends on $t,\varphi$
only via $\varphi-t/\C,$ hence we may quantize it to a fiber invariant operator $A
\in \Psibf^0(X)$.

Note that on $\supp \chi_1(\bullet),$ $\xibhat\leq \delta$.  Meanwhile, on $\supp
\chi_0(\bullet),$ $\omega<\beta^2 \delta^2 +\beta^2 \delta \xibhat;$ owing
to the support constraints on $\xibhat,$ then, $\omega<2\delta^2
\beta^2$ (and thus both $r$ and $\abs{\varphi - t/\C-\varphi_0}$ are bounded above by $\sqrt{2}\delta\beta$). On $\supp
\chi_0(\bullet)$, $\xibhat\geq-\delta$, so overall $\xibhat \in
[-\delta,\delta]$ on $\supp a.$  Note also, for later use, that
$\abs{1-\phi/\delta}<4$ on $\supp a$.

The parameter $\beta$ will later be used to obtain
sufficient positivity of commutator terms, and may need to be taken
large; thus, note that \emph{for any $\beta \in (0, \infty)$ there
  exists $\delta>0$ such that $\supp a \cap \{\xibhat>0\} \subset U.$}
We alert the reader that at the moment when $\beta$ is taken
large in the following argument, we simultaneously adjust
$\delta=\delta(\beta)$ to maintain the desired support properties.

Now applying Lemma~\ref{lemma:maincommutator} (as well as recalling that $[\pert,
A^*A] \in \Psibf^{0}(X)$ since $\pert$ is a uniform b-operator of order $1$)
and pairing with $u$
yields
\begin{equation} \label{keyterms}
\begin{aligned}
2 \Im \ang{f, A^*A u} &= \la i [P,A^*A] u, u \ra + \la i(P - P^*)A^*A u, u \ra \\
 &=\ang{(D_r^* L_r D_r +(r^{-1} \fib) L_\fib (r^{-1} \fib)+ D_r^* L' +
L'' D_r + L_0 )u,u}\\
	&\quad +\la i[A^*A, \pert]  u, u \ra+ \la i (P-P^*)A^*A u,u \ra \\
&= \ang{(D_r^* L_r D_r +(r^{-1} \fib) L_\fib (r^{-1} \fib)+ D_r^* L' +
	L'' D_r + L_0 + \tilde{\pert}A^*A  )u,u}\\
&\quad + \la  R_{0}' u,u \ra 
\end{aligned}
\end{equation}
where the $L$ operators are defined as in
Lemma~\ref{lemma:maincommutator}, $\tilde{\pert} = i(\pert-\pert^*)
\in \Diff_\bu^1(X)$, and there are further non-fiber-invariant terms $R'_0 \in \Psib^0(X)$ arising from commuting with $\pert$.

Now let $B \in
\Psibf^{-1/2}$ have symbol $$b=2\abs{\taub}^{-1/2} \delta^{-1/2} (\chi_0
\chi_0')^{1/2} \chi_1 \chi_2 \chi_3.$$  This is the quantity appearing in
commutator terms with derivatives hitting the $\chi_0^2$ factor of
$a^2;$ its support is that same as $\supp a$.  In particular,
$$
\sigmab{-1}(L_r)=\sigmab{-1}(L_\fib)= 4 a \pa_{\xib} a
=b^2+e'+e''
$$
where $\supp e'\subset \{\xibhat>0\},$ and $\supp e'' \cap
\Tbdotstar X=\emptyset:$ the $b^2$ term comes, as noted above, from
the derivative falling on $\chi_0,$ while the term with a derivative
falling on $\chi_1$ gives $e'$ and the term with a derivative on
$\chi_2$ gives $e''.$  Note, for later use, that owing to the specific
form of the cutoff function $\chi_0,$ we have arranged that $a^2
b^{-1}$ is a smooth $\fib$-invariant symbol of order $1/2.$

Thus,
\begin{multline}
D_r^* L_r D_r +(r^{-1} \fib) L_\fib (r^{-1} \fib)=D_r^* B^*B D_r +(r^{-1} \fib) B^*B (r^{-1} \fib) +D_r^* (E'+E''+R_{-2}) D_r \\+(r^{-1} \fib) (E'+E''+R_{-2}) (r^{-1} \fib)
\end{multline}
where
\begin{itemize}
\item
  $\displaystyle E' \in \Psibf^{-1}(X),\ \opWFb(E') \subset \xibhat^{-1}((0,\infty))
$
\item
  $\displaystyle \displaystyle E'' \in \Psibf^{-1}(X),\ \opWFb(E'') \cap\Tbdotstar X=\emptyset.$
  \item $\displaystyle R_{-2} \in \Psibf^{-2}(X).$
  \item
    And where $E',E'', R_{-2}$ will now mean potentially \emph{different} operators
    with the above properties in each occurrence of these symbols.
\end{itemize}

Now we turn to the $D_r, D_r^*$ terms in \eqref{keyterms}.  We
have $\sigmab{0} (L')=\sigmab{0}(L'')=2 a \pa_r a,$ and since the only
$r$-dependence of $a$ is in the $\chi_0(1-\phi/\delta)$ factor and $\chi_0\chi_0'\chi_1^2\chi_2^2\chi_3^2 = \frac{1}{4}|\tau|\delta b^2$ we may
write these symbols as 
$$
-\frac{1}{2}\abs\taub \delta b^2 \pa_r \phi=-\abs\taub\frac{1}{\beta^2 \delta} r b^2.
$$
Recall that $r< \sqrt 2 \delta \beta$ on $\supp b,$ so that in fact
$$
\sigmab{0}(L')=b^2 f_1'
$$
for $f_1' \in S^{1}$ with $$\abs{f_1'} \leq \sqrt 2 \abs\taub \beta^{-1},$$ supported on
any desired open neighborhood of $\supp b.$  Analogously,
$\sigmab{0}(L'')=b^2 f_1''$ with the same properties.

Likewise, the symbol of the $L_0$ term involves only $t$ derivatives falling
on $\chi_0,$ hence
$$
\sigmab{1}(L_0)=-\taub \abs\taub b^2 \pa_t \phi=2\C^{-1}\taub \abs\taub\frac{1}{\beta^2 \delta} (\varphi-t/\C-\varphi_0) b^2;
$$
Again we thus find
$$
\sigmab{1}(L_0)=b^2 f_2
$$
where $f_2$ is estimated by
$$
\abs{f_2} \leq 2 \sqrt 2 \C^{-1} \abs{\taub}^2\beta^{-1},
$$
since $\abs{\varphi-t/\C-\varphi_0}<\sqrt{2} \beta\delta$ on $\supp b.$

We now consider the $\tilde{\pert}A^*A$ term, which is also first order. Note that due to the relationship $\chi_0'(s) = s^{-2}\chi_0(s)$ we have
	$$ a^2 = \frac{1}{2} |\tau|\delta\left(1-\frac{\phi}{\delta}\right)b^2, $$
and as noted above, $\left|1 - \phi/\delta \right| \le 4$ on the support of $b$. Then taking $\sigmab{1}(\tilde{\pert}) = \upsilon$ we have
	$$ \sigmab{1}(\tilde{\pert} A^*A) = \frac{\upsilon}{2} |\tau|\delta \left(1-\frac{\phi}{\delta}\right)b^2 = \tilde{f}_2b^2 $$
where
	$$ |\tilde{f_2}| \le 2|\upsilon \tau| \delta. $$ 
In the following argument we will pick $\beta$ large to make $f_1'$
and $f_2$ above sufficiently small. As noted earlier, there is then a
$\delta(\beta)>0$ for which $a$ has the requisite support properties
whenever $\delta\leq\delta(\beta)$. To
make $\tilde{f}_2$ sufficiently small, we pick $\delta \le
\delta(\beta)$.

Assembling this information yields \begin{align*}L' &=B^* B
                                                      F_1'+R_{-1},\\
                                     L'' &=B^*B
F_1''+R_{-1},\\ L_0 + \tilde{\pert}A^*A &=B^*BF_2+R'_0\end{align*} where $F_1',\ F_1''$ have symbols
$f_1',f_1'',$ $F_2$ has symbol $f_2+ \tilde{f_2},$ where $R_s \in
\Psibf^{s}(X)$, and where $R_0' \in \Psib^0(X)$ (with the notation recycled to indicate a different
remainder term in each case).

Thus, absorbing further lower order commutator terms in the
ever-changing $R_\bullet$ terms below and including a $(1/2)B^*B \pert$
term into $R_0'$ gives
$$
\begin{aligned}
2\Im &\ang{f, A^*A u}\\
&= \big\langle\big(B^*B (D_r^* D_r + r^{-2} \fib^2)+D_r^* (E'+E''+R_{-2})
  D_r \\
  	&\quad +(r^{-1} \fib) (E'+E''+R_{-2}) (r^{-1} \fib) 
+  R_{-2} r^{-1} \fib + D_r^*(B^*B F_1'+R_{-1})\\ &\quad+B^*B F_2
+R'_0\big) u,u\big \rangle\\
&= \big\langle\frac 12 B^*B (D_r^* D_r + r^{-2} \fib^2 +D_t^2-P)+D_r^* (E'+E''+R_{-2})
  D_r \\ &\quad+(r^{-1} \fib) (E'+E''+R_{-2}) (r^{-1} \fib)
+R_{-2}r^{-1} \fib  + D_r^*(B^*B F_1'+R_{-1})  \\&\quad +B^*B F_2
+R'_0\big) u,u\big \rangle\\
&=
(1/2)\big(\norm{D_t Bu}^2+\norm{r^{-1} \fib B u}^2 + \norm{D_r B u}^2
- \ang{f, B^* B u}\big) \\ &\quad + \ang{(E'+E''+R_{-2})D_r u, D_r u}
+\ang{(E'+E''+R_{-2}) (r^{-1} \fib)u , r^{-1} \fib u}\\
&\quad + \ang{B F_1'u,
D_r B u} +\ang{B F_2u,B u} 
+\ang{R_{-2}r^{-1}\fib u, u} + \ang{R_{-1}u,D_r  u} + \ang{R'_0u,u}
\end{aligned}
$$
Note that for brevity, we have dropped terms of the form $(\bullet) D_r$ in the
first line
favor of terms $D_r^*(\bullet),$ since they have the same imaginary
part modulo commutator terms absorbed elsewhere.  As before, $R_0'$
refers to a non-fiber-invariant operator.

We now proceed with our inductive argument.  Suppose that
$$
\varrho\notin\WFbhu^m u
$$
with $m\leq s-1/2;$ it will suffice to show that
$$
\varrho\notin\WFbh^{m+1/2} u.
$$

To this end we shift around the orders in the commutator computation as follows:
for $\ind \in (0,1),$ fix 
$$
Q^s_\ind \in \Psi^{s}(\RR)
$$
given by left quantization of
$$
\ang{\taub}^s (1+ \ind\abs{\taub}^2)^{-s/2},
$$
and commuting with $\Box_g.$
Thus, $Q^s_\ind$ is a family that is uniformly bounded in
$\Psi^s(\RR),$ convergent in $\Psi^{s+\epsilon}(\RR)$ to $\ang{D_t}^s$
modulo a fixed, $s$-dependent smoothing operator.  We would like to
view $Q^s_\ind$ as lying in the calculus $\Psibf(X)$ but the mild
technical snag is that $q$ is not a symbol in $(\xi,\tau,\eta).$
However, it is such a symbol when cut off away from $\xi=\eta=0$,
hence setting
$$
A_{m+1}\equiv  A Q_\ind^{m+1} 
$$
we easily see that if $A=\tOpb (a)$ then $A_{m+1}=\tOpb(aq)$ and lies in
our calculus, hence we may treat it for all practical purposes as
lying in the calculus.\footnote{Similar considerations famously occur in the study of
$D_t-\sqrt{\Lap_g}$ on $\RR\times M$ with $M$ a compact manifold:
$\sqrt{\Lap_g} $is not a pseudodifferential operator on the product,
but this is of little importance;
see, e.g., \cite{DuGu:75}.}
We omit the index $\ind$ from the notation in order to keep it
uncluttered, but here and henceforth on we decorate operator families
bounded in $\Psibf^s(X)$ with the index $s$ as a bookkeeping aid.  We
will use consistent notation for the operator $B$ letting $B_{m+1/2}$
denote the family $Q^{m+1}_\ind B$ (hence again indexing the family by the
order of the operator space in which it is bounded).
 Note that the operators
$F_\bullet$ previously had a subscript referring to the index, and
these remain unchanged (indeed, they are not families).  The $E',E'',R$
operators below are replaced by families as well, with the same
indexing convention employed.

We then have 
$$
\begin{aligned}
2\Im &\ang{f, A_{m+1}^*A_{m+1} u}= \\
&=
(1/2)\big(\norm{D_t B_{m+1/2}u}^2+\norm{r^{-1} \fib B_{m+1/2} u}^2 + \norm{D_r B_{m+1/2} u}^2  \\ &\quad
-\ang{f, B_{m+1/2}^* B_{m+1/2} u}\big) + \ang{(E_{2m+1}'+E_{2m+1}''+R_{2m})D_r u, D_r u}  \\ &\quad
+\ang{(E_{2m+1}'+E_{2m+1}''+R_{2m}) (r^{-1} \fib)u , r^{-1} \fib u}  \\ &\quad + \ang{B_{m+1/2} F_{1}'u,
D_r B_{m+1/2} u} + \ang{B_{m+1/2} F_{2} u,B_{m+1/2} u}  \\ &\quad+ \ang{R_{2m+1}u,D_ru}  + \ang{R_{2m}u,r^{-1}\fib u}
+\ang{R'_{2m+2}u,u},
\end{aligned}
$$
again with $R'_{2m+2} \in \Psib^{2m+2}(X)$, now including terms
arising from $[\pert, Q_\ind]$.  We
will treat the first three terms on the RHS as the main positive
terms, either absorbing the rest of the terms into them or else
estimating them by induction or regularity hypothesis.

By elliptic regularity, i.e.\ the quantitative statement (by the
closed graph theorem or inspection of the proof) obtained from Proposition~\ref{prop:ellipticreg},
$$
(1/2)\big(\norm{D_t B_{m+1/2}u}^2+\norm{r^{-1} \fib B_{m+1/2} u}^2 + \norm{D_r B_{m+1/2}
  u}^2\big)\geq c \norm{B_{m+1/2}u}^2_\Hone
$$
for some $c>0$.
Owing to the symbol estimates on the $F$ terms, by taking $\beta$
sufficiently large and $\delta$ sufficiently small, we may apply Cauchy--Schwarz to estimate
$$
\abs{ \ang{B_{m+1/2} F_1'u,
D_r B_{m+1/2} u}+\ang{B_{m+1/2} F_2u,B_{m+1/2} u}} \leq (c/2)\norm{B_{m+1/2}u}^2_\Hone,
$$
thus allowing us to absorb these terms in the main positive term.  Meanwhile,
since $\sigmab{}(A_{m+1}^*A_{m+1})^2/\sigmab{}(B_{m+1/2})\equiv g \in S_\fib^{m+3/2},$ we may apply Cauchy--Schwarz
and the b-symbol calculus to estimate
$$
\abs{\ang{f, A_{m+1}^*A_{m+1} u}} \leq \norm{G_{m+3/2}
    f}_{\Hmone}\norm{B_{m+1/2} u}_\Hone+ \norm{R_{m+1}
    f}_{\Hmone}\norm{R_{m} u}_{\Hone}
$$
with $R_{s} \in \Psibf^{s}(X),$ microsupported in an arbitrary
neighborhood of $\WFbf' B.$  Since by assumption, $m+3/2
\leq s+1$, the $G$ term above is bounded by the assumption that
$\varrho \notin \WFbhstaru^{s+1} f;$ the second term on the RHS is
bounded by the same estimate on $f$ and by the inductive assumption
$\varrho \notin \WFbh^m u.$  The lower order term $\ang{f,
  B_{m+1/2}^* B_{m+1/2} u}$ is likewise easily estimated in the same manner.

Turning to other terms, we find that the $E'$ terms are uniformly
bounded as $\ind \downarrow 0$ by elliptic regularity; the $E''$ terms
are uniformly bounded by our wavefront set hypothesis, since they are
(uniformly) microsupported in the control region $\xib>0;$ all
$R_\bullet$ terms are uniformly bounded by our inductive assumption
$\varrho \notin \WFbh^m u;$ the $R'_{2m+2}$ term is estimated by the
inductive assumption and Corollary~\ref{cor:uniformbounded}.

Putting together the above observations (and lumping the bounded terms
described above, with the exception of $\norm{Gu}$, into a single constant) yields
$$
(c/2) \norm{B_{m+1/2} u}_\Hone^2 \leq \norm{G_{m+3/2}
    f}_{\Hmone}\norm{B_{m+1/2} u}_\Hone+C,
  $$
  hence by a further Cauchy--Schwarz, boundedness of $\norm{G_{m+3/2}
    f}_{\Hmone}$ yields uniform boundedness of $\norm{B_{m+1/2}
    u}_\Hone$ as $\ind\downarrow 0.$
 A standard compactness argument now implies
that
$$
\varrho \notin \WFbhu^{m+1/2} u,
$$
and the inductive step is complete.
\end{proof}

Before proceeding to employ our propagation result to obtain an
existence theorem, we record a corollary
that is essentially just a quantitative restatement.

Let $\proj: \Tbdotstar X \to X_\fib$ be the projection map from the
compressed b-cotangent bundle onto the fiber-quotiented base space. 

\begin{corollary}\label{cor:quantitative}
Fix any constants $r_1>r_0>0,$  and let $\crap \in \Psibf^0(X)$ be elliptic
on $\fcal_{I} \cap \proj^{-1}(\{r \in [r_0, r_1]\}).$  There exists $B
\in \Psibf^0(X)$, elliptic on $\Tbdotstar_{\pa X} X_\fib$, such that for
all $N$ there exists $C$ such that
\begin{equation}\label{quantPoS}
\norm{B u}_{\HbHone^s} \leq  C\big(\norm{\crap u}_{\HbHone^s} +  \norm{u}_{\HbHone^{-N}} + \norm{Pu}_{\HbHonestar^{s+1}}\big)
\end{equation}
  \end{corollary}
This result can of course be localized in fibers and microlocalized in
$\tau$, but this global version is all we need below.\begin{proof} Ordinary propagation of singularities in the
interior allows us to take $r_0$ as small as desired.  By the
observations in
item~\ref{propagationresult} of Remark~\ref{rem:big}, control of
$\crap u$
plus elliptic regularity implies that the hypotheses of the propagation theorem
are fulfilled (with appropriate choice of signs) at every $\varrho$, i.e., we have the desired regularity  for all values of
$\varphi_0$ and $\tau_0$, either in the region $\xi>0$ for $\tau_0>0$ or
$\xi<0$ for $\tau_0<0.$  Thus, the theorem (technically either with
the closed graph theorem applied to obtain the quantitative statement
from the qualitative one, or else from direct examination of the
quantitative estimate that proves the theorem) yields the estimate
\eqref{quantPoS} where $B$ is microsupported near any desired
$\varrho\in \Tbdotstar_{\pa X} X_\fib$.  Since the cosphere bundle of
$\Tbdotstar_{\pa X} X_\fib$ is compact (with coordinates $\varphi_0
\in S^1,$ $\hat \tau \in \pm 1$), we can sum up these estimates
and assume $B$ is elliptic on all $\Tbdotstar_{\pa X} X_\fib$.
\end{proof}

\section{Causal solutions to the wave equation}\label{sec:causal}

We now consider the problem of finding causal solutions to the wave
equation: given an appropriate $f,$ we wish, e.g., to find $u$ such that
$$
Pu=f
$$
in the sense of distributions and such that the support of $u$ is in
the asymptotically forward-in-time Hamilton flowout of the support of $f.$
Recall the notion of forward-in-time flowout is not very well-defined
locally, as $t$ is not monotone along some bicharacteristics. However, along
every bicharacteristic not reaching $r=0$, $t$ is
\emph{eventually} monotone along the flow. Whether $t$ is
aymptotically increasing or decreasing on such bicharacteristics is
determined by the sign of $\taub$, a conserved quantity.

Thus for a set $S \subset \Sigma,$ let $\Phi_\pm(S)$ denote the maximally
extended generalized flowout of $S$ in the asymptotically
forward/backward time-direction, over $X^\circ$:
$$
\Phi_{\pm}(q)=\bigcup_{s \cdot \sgn \tau(q) \gtrless 0} \Phi_s(q)
$$
where on incoming/outgoing rays we take $\Phi_s(q)$ to be undefined for parameter values beyond
those where it reaches $r=0.$

We will not trouble to
define the broken flow across the string at $r=0,$ since the global
appearance of singularities there seems unavoidable.

The goal stated above turns out to be too much to ask: we cannot
expect nontrivial solutions supported in $t>T_0$ for any $T_0$, e.g.\ since
the mode-by-mode equation obtained by separation of variables in
$\Box_g$ is
\emph{elliptic} for $r<\C$ and hence enjoys unique continuation.
(See
\cite{MoWu:21} for a discussion of 
the angular mode equation.)
So we accept instead a solution on
a(n arbitrarily large) compact set with the minimal wavefront set that
the propagation theorem above would permit: the forward flowout of
$\WF f$ together with the forward flowout of the string itself, for
all $t.$  (Again, cf.\ \cite{MoWu:21}, where the authors performed similar analysis for
angular mode solutions; this simplified setting did not allow the excitation of
singularities emerging from the string, however.)

Recalling that
$$
P=\Box_g+\pert
$$
where $\pert \in  \Diff_{bu}^1(X)$ is a first-order perturbation, we now strengthen our
hypotheses on $\pert$ in order to obtain some necessary unique
continuation results for the perturbed equation.  We offer two different auxiliary hypotheses on $\pert$
that yield stronger solvability results.

\begin{hyp}{$\mathcal{C}^\omega$}\label{hyp:analytic}
$\pert$ has analytic coefficients.
  \end{hyp}

  \begin{hyp}{$(t,\varphi)$}\label{hyp:tphi}
$\displaystyle [\pa_t,\pert]=[\pa_\varphi, \pert]=0.$
\end{hyp}



\begin{theorem}\label{theorem:forward}
  Given compact sets $K_0\subset K \subset X$ with  $K_0 \subset K^\circ$,
  there exists a
finite dimensional space $N \subset H_b^\infty(X)\cap \E'(X)$ such that if
$f \in \dot{H}_b^{m}(K_0)\cap N^\perp$ ($L^2$ orthocomplement)
there exists $u \in H_b^{m+1}(K^\circ)$ with
$$
Pu=f\quad \text{ on } K^\circ.
$$
Furthermore, $u$ satisfies
\begin{equation}\label{WFrelation}
\WFb u\backslash \WFb f \subset \fcal_O\cup \Phi_+ (\WFb
f\cap \Sigma).
\end{equation}

The solution $u$ is unique on $K^\circ$ modulo a distribution $w$ with
$\WFbh w \subset \fcal _O.$

If $\pert$ additionally satisfies either Hypothesis~\ref{hyp:analytic} or \ref{hyp:tphi}, then a (unique in the above sense)
solution exists for \emph{all} 
$f \in \dot{H}_b^{m}(K_0)$.
\end{theorem}

\begin{remark}\mbox{}\begin{enumerate}\item
Here in the existence theorem we have dropped the refinement of Sobolev spaces based on $\Hone$
in favor of the simpler $H_b^m$ spaces.  (We do still need the refined
notion of wavefront set in employing our propagation statements.)  We
will use results obtained above in the setting of the $\Hb^m$ scale
of spaces, and remind the reader (cf.\ Remark~\ref{remark:Sob}) that
the distinction between $H_b^m$ and $\Hb^m$ is moot over compact sets.

\item
More precise statements than the basic ``solvability in $H_b^{m+1}$ for
data in $H_b^{m}$'' in fact hold, using $\Hone$-based spaces.  So, for instance, the method used
here shows that
for every
$f \in \HbHonestar^{-1}$ with  $\supp f
\subset \{r<R_0\},$ there exists $u \in (\HbHonestar^2\cap L^2)^*$
solving the equation as above.  Such estimates have the virtue of treating radial and
fiber-derivatives differently from arbitrary b-derivatives.
Likewise, we note that we can obtain the more precise statement that the set of
obstructions $N$ is contained in $\HbHone^\infty\cap \E'(X).$

\item
Even in the argument as given below in terms of b-Sobolev spaces, the
reader will note that our hypotheses on $f$ are in fact that it lie
the dual space of an \emph{intersection} of Sobolev spaces. This would enable us
to give alternate hypotheses on $f$ involving higher regularity
but a more singular weight at $r=0.$

\item
A more precise version of the wavefront relation follows from our
elliptic regularity statements and interior propagation of singularities:
\begin{equation}\label{betterWFrelation}
\WFbh u\backslash \WFbhstar f \subset \fcal_O\cup \Phi_+ (\WFbhstar f\cap \Sigma).
\end{equation}

\item 
The assumption that $f$ is a compactly supported distribution is needed to obtain the uniqueness result but not the existence result, which merely requires $f$ be an extendible distribution defined on $K_0$.
\end{enumerate}
\end{remark}
We now prove Theorem~\ref{theorem:forward}.
Assume without loss of generality (expanding and time-translating $K$
if necessary) that
$$
K = \{r\leq R_0,\ t \in [0, T]\},
$$
and moreover that $R_0>\abs{\C}.$  Note that this means that for $r>R_0,$
under the bicharacteristic flow (in the b-cotangent variables)
$$
\dot t = r^2\taub - \C (\C\taub +\etab)
$$
has the same sign as $\taub$ on the characteristic set, hence $t$ is
monotone along the null bicharacteristic flow in this region.

Let $\proj: \Tbstar X \to X$ be the projection map from the b-cotangent bundle onto the base space.

\begin{lemma}\label{lemma:bichars}
  There exists $R> R_0$ as above and $T' \gg 0$ so that for each null
  bicharacteristic $\gamma(s)$ with $\proj(\gamma(0)) \in K$ and $\gamma(0) \notin \fcal_O,$ there exists
  $s_0$ with
  \begin{enumerate}
  \item \label{flow1} $R+1<r(\gamma(s_0))<2R$
  \item \label{flow2} $-T'<t(\gamma(s_0))<t(\gamma(0))$
  \item \label{flow3} $\displaystyle\frac{\xib}{r \taub}\rvert_{\gamma(s_0)} > 3/4$
\item \label{flow4} For $s$ between $0$ and $s_0$,
  $\displaystyle\proj(\gamma(s))\in \{\abs{t}<T',\  r<2R\}$.
    \end{enumerate}
  \end{lemma}
  Thus any bicharacteristic from $K$ not emerging from the string comes from the
  region $r>R+1$ at backward time, and at that time has a large radial
  component.  It is, moreover,
  oriented in an incoming
  direction (property \eqref{flow3}), as it must reach $K,$ where $r<R_0$, as time increases.
  This must occur within some fixed amount of elapsed backwards time
  $T'$ (which of course depends on the size of $K$) and within distance
  $2R$ of the string.

  We also note for later use the simpler fact that sufficiently far back in time, the
  geodesic lies in $\{r>2R\}$.

  \begin{proof}
To begin, we employ the Minkowski coordinates $(t'=t-\C \varphi,x)$
of Section~\ref{sectionBichars} near any desired bicharacteristic.  Recall that the null geodesic
flow is then (up to reparametrization) the lift of the Minkowski geodesic flow
$$
x=x_0+sv,\ t'=t'_0\pm s,
$$
where $\abs{v}=1$.  We use the notation $t(s),$ $x(s)$ instead of
$t(\gamma(s)),$ $x(\gamma(s))$, etc.

Choose $R$ sufficiently large so that
\begin{equation}\label{R}
\max\{R_0+R+1, 2 \abs{\C} \pi\}< 2R-R_0.
\end{equation}
If we choose $s_0$ with  \begin{equation}\label{s} -2R+R_0<s_0<-\max\{R_0+R+1, 2 \abs{\C} \pi\},\end{equation} then provided $\abs{x_0}<R_0$,
\begin{equation}\label{requation}
r(s_0)=\abs{x(s)} \geq \abs{s_0} -\abs{x_0} >(R_0+R+1)-R_0=R+1.
\end{equation}
Likewise,
$$
r(s_0) \leq \abs{s_0}+\abs{x_0} <(2R-R_0)+R_0,
$$
establishing \eqref{flow1}.

Meanwhile, it is always the case that
$$
t\in [t'-\abs{\C} \pi, t'+\abs{\C} \pi],
$$
hence, since the flow in $t'$ is likewise very simple, subject to
\eqref{s} we also certainly have
$$
t(s_0)\leq t'(s_0)+\abs{\C} \pi< (t'_0-2 \abs{\C} \pi)+\abs{\C} \pi\leq t(0).
$$
On the other hand,
$$
t(s_0)\geq  t'(s_0)-\abs{\C} \pi > (t'_0-(2R-R_0))-\abs{\C} \pi=-T'
$$
where
$$
T' \equiv 2R-R_0+\abs{\C} \pi.
$$
this establishes \eqref{flow2}.  Note that all the estimates above
required was for $R$ to satisfy the inequality \eqref{R}, hence we may freely increase $R$
in what follows.  Note also that increasing $R$ increases $T'$.

Finally, we turn to \eqref{flow3}, which requires examination of the
cotangent variables. Assume $s_0$ satisfies \eqref{s}.  Then
$$
\begin{aligned}
  -\dot{r}&=-\frac{\ang{x_0 +s_0 v,v}}{\abs{x_0+s_0v}}\\
  &=-\frac{\ang{x_0, v}+s_0}{\abs{x_0+s_0v}}\\
  &\geq \frac{\abs{s_0}-R_0}{\abs{s_0}+R_0}\\
  &> \frac{1-R_0/(R_0+R+1)}{1+R_0/(R_0+R+1)};
\end{aligned}
$$
increasing $R$ as needed, we can thus ensure that for
$s_0$ satisfying \eqref{s},
$$
\dot r \leq-0.9.
$$
Likewise, since $r(s_0)>R+1$ and $|\dot{r}|\le 1$ since $x(s)$ is parametrized at unit speed, 
$$
\begin{aligned}
  \big \lvert \frac d{ds} \hat x\big \rvert &= \big \lvert\frac d{ds}
  \frac{x}{\abs{x}} \big \rvert\\
  &= \big \lvert-\frac{\dot r}{r^2} x +\frac{v}{r}\big \rvert\\
  &\leq 2(R+1)^{-1}\\
  &\leq 0.01/\abs{\C}
\end{aligned}
$$
provided $R$ is sufficiently large.  Hence we certainly have
$\abs{\dot \varphi}<0.1/\abs{\C}$ for such $s_0$, and, finally, since $\dot t'=1$,
$$
\dot t= \dot t'+ \C \dot \varphi<1.1
$$
at these points.
Thus
\begin{equation}\label{drdt}
\frac{dr}{dt} =\frac{\dot r}{\dot t}< -\frac{0.9}{1.1}<-0.8.
\end{equation}

Finally, note from \eqref{bHam} that under the
rescaled Hamilton flow employed there (since the change of scaling factor cancels out in numerator and denominator)  
$$
\begin{aligned}
  \frac{dr}{dt} &= \frac{-\xi r}{r^2 \tau-\C (\C \tau + \eta)}\\
  &= -\frac{\xi}{r\tau(1-r^{-2}\C (\C+\hat \eta))}
\end{aligned}
$$
On the characteristic set over $K$, $\abs{\C+\hat  \eta}<R_0,$ so certainly $\abs{\C (\C+\hat \eta)}$ is bounded by
some number $\digamma$ on the characteristic set over $K$.  Thus by \eqref{drdt}
$$
\frac{\xi}{r\tau(1-r^{-2}\C (\C+\hat \eta))}=-\frac{dr}{dt} > 0.8
$$
hence when $r>R+1$,
$$
\frac{\xi}{r\tau}>0.8 (1-r^{-2}\C (\C+\hat \eta))>0.8(1-(R+1)^{-2}\digamma)
$$
and taking $R$ large enough ensures that \eqref{flow3} holds.

We now turn to condition \eqref{flow4}.  To clarify the exposition,
assume without loss of generality that $s_0<0$---the reverse case
merely involves overall sign changes.  Since $\ddot{r}>0$ along the
null bicharacteristic flow, $r$ must achieve its maximum on the
interval $[s_0,0]$ at $s_0$, hence cannot have exceeded
$2R$.

As for the $t$ variable, as noted in
Section~\ref{sectionBichars} it is monotone increasing along the flow
when $r>\abs{\C}$, and it cannot ever exceed $t(0)+2\abs{\C}\pi$ for
$s\leq 0$.  Its maximum
along the flow from $K$ with parameter $s\in [s_0, 0]$ can thus be no larger than $T+2\abs{\C}\pi$; we increase 
$R$ as needed to ensure that $T'=2R-R_0 +
\abs{A}\pi$ is larger than this quantity.

Now note that $r$ must exceed $R_0>\abs{\C}$ whenever $s<-2R_0$
  (cf.\ \eqref{requation} above).  Since
  $$t'(s)=t'(0)+s = t(0)-\C \varphi(0)+s,$$ on the interval
  $s\in [-2R_0, 0]$, $t' \geq -2R_0-\abs{\C}\pi$ hence
  $t \geq -2R_0-2\abs{\C} \pi$.  Thus, if $t$ achieves its minimum on
  $[-s_0, 0]$ for $s\geq -2R_0,$ that minimum must be at least
  $-(2R_0+2\abs{\C}\pi)$.  Taking $R$ sufficiently large, we ensure
  that $T'>2R_0+2\abs{\C}\pi$, hence if the minimum of $t$ on $[s_0,
  0]$ arises for
  $s\in [-2R_0, 0]$, we certainly have $t>-T'$ on $[s_0, 0]$.  The
  same holds if the minimum is not achieved on $[-2R_0, 0]$, since
  then it must arise at $s=s_0$ by monotonicity of $t$ when
  $r>\abs{\C}$.  Thus the $t$ variable stays within $(-T',T')$ on
  $[-s_0,0]$ as desired.
\end{proof}

  
  Recall that $K \subset \{t \in [0, T]\}$.
  Let $K'$ denote the compact set
\begin{equation}\label{Kprime}
K'=\{r\leq 2R,\ t \in [-T', T']\}\subset X.
\end{equation}
  
Now we construct a complex absorbing potential $W \in \Psibf^2(X),$ supported away
from $r=0,$ such that
\begin{enumerate}
\item $W$ is elliptic on the set
  $$
\IC \equiv \{r>R+1\} \cap \left\{\frac{\xib}{r
  \taub}>3/4 \right\}.
  $$
\item $\sgn \sigmab{2}(W) = \sgn \xib=\sgn\taub$ on $\IC.$
\item The Schwartz kernel of $W$ is supported on $\{r>R\}^2.$
  \item $[\pa_t, W]=[\pa_\varphi, W]=0.$
\end{enumerate}
Note in particular that the elliptic set of $W$ contains the incoming
points $\fcal_I$ sufficiently far from the string, since on this set $(A\taub+\etab)=0$ and $\xib/(r\taub) = 1$.
We now let
$$
\tP\equiv P-iW
$$
be the wave operator modified by the complex absorbing potential.

We now state a proposition in which we bring our propagation of
singularities results to bear in a global fashion.
In what follows, we use the notation $\tP^{(*)}$ to denote an operator
that may be taken to be \emph{either} $\tP$ \emph{or} $\tP^*.$
  \begin{proposition}\label{prop:Fredholm1}
For any $s \in \RR,$ $N \in \NN,$ there exists $C$ such that for all $\phi \in \dot\Hone(K')$,
\begin{equation}\label{lowerbound}
\norm{\phi}_{\HbHone^s} \leq C \norm{\phi}_{\HbHone^{-N}}+ C \norm{\tP^{(*)} \phi}_{\HbHonestar^{s+1}}.
\end{equation}
\end{proposition}
Recall that the conventions on the Sobolev spaces are such that, away
from $r=0,$ the norm on the LHS is $H^{s+1}$ and the norm of the
$\tP^{(*)}$ term on the RHS is $H^s.$
    \begin{proof}
     
We prove the result for $\tP,$ with the result for $\tP^*$ being
analogous, with reversed signs (and reversal in the direction of
propagation of singularities on each bicharacteristic).

We claim that for any $q \in \Tbstar X$ with $\proj(q) \in K'$ there
exists $A_q \in \Psibf^0(X)$ elliptic at the fiber through $q$ such that
	\begin{equation}\label{localest} 
		\norm{A_q \phi}_{\HbHone^s} \leq  C \norm{\phi}_{\HbHone^{-N}}+ C
	\norm{\tP \phi}_{\HbHonestar^{s+1}}. 
	\end{equation}
Adding up these regularity estimates, using compactness of the cosphere bundle $\Sbstar K'_\fib$, yields the desired global estimate over $K'$. 

Thus it is left to prove the claim. To this end, let $q \in \Tbstar X$ with $\proj(q) \in K'$ be given.  

If
$\proj(q) \in \pa X$ and $q \notin \Tbdotstar X$, then by our elliptic
regularity estimate, Proposition~\ref{prop:ellipticreg}, there exists
$A_q\in \Psibf^0(X)$ elliptic at $q$ such that
$$
\norm{A_q \phi }_{\HbHone^s} \leq C \norm{\phi}_{\HbHone^{-N}} + C
\norm{\tP \phi}_{\HbHonestar^s},
$$
which being an \emph{elliptic} estimate, is a stronger estimate than needed on $P u$, and indeed implies \eqref{localest}.
Likewise, if $\proj (q) \in X^\circ$ and $q \notin \Sigma$, an
estimate of the same form applies by standard microlocal elliptic
regularity.

If by contrast, either $\proj(q) \in X^\circ$ and $x \in \Sigma$ or else
$q \in \Tbdotstar_{\pa X} X$, then we will apply propagation
estimates.  First consider the case where $q \notin \fcal_O$ (so
$\proj (q)
\in X^\circ$).  Let us suppose for the moment that $\tau(q)>0$.  Recall that $\tau$ is conserved along
the null bicharacteristic flow and asymptotically we have
$\sgn \dot{t} = \sgn \tau >0$ so that the null bicharacteristic
satisfying $\gamma(0)=q$ is asymptotically forward-oriented in
time. Furthermore, since geodesics not in $\fcal_O$ escape to infinity
as $t \to -\infty$, there is a point
$\gamma(s)$ \emph{backwards} along the flow (relative to the curve parameter)
where $\gamma(s)\notin \Tbstar K'$.
 Due to our
choice of sign for the principal symbol of $W$, regularity propagates
\textit{forward} along the flow (and singularities propagate backward)
by the results of \cite[Section 2.5]{Va:13}. Thus for some
$A_q \in \Psibf^0$ elliptic near $q$, we have
\begin{equation}\label{local-est}
	\norm{A_q\phi}_{\HbHone^s} \leq C \norm{\phi}_{\HbHone^s((K')^c)}+C \norm{\phi}_{\HbHone^{-N}}+ C \norm{\tP \phi}_{\HbHonestar^{s+1}},
\end{equation} 
 (cf.\ \cite[Equation 2.18]{Va:13}), and the first term on the right
 hand side vanishes since $\supp \phi \subset K'$, hence equation \eqref{localest} holds.
 The same argument
applies, mutatis mutandis, under forward flow if $\tau<0$: the
geodesic still escapes to infinity as $t\to-\infty$ but now this is
forward along the flow, hence the same propagation estimates hold by
our choice of the sign of $\sigma_b^2(W)$, which has been arranged so
propagation of regularity is forward in $t$, no matter which sign of $\tau$ is chosen.

Now we turn to the case $q \in \Tbdotstar_{\pa X} X$.
Corollary~\ref{cor:quantitative} yields, for an $A_q$ elliptic near $q,$
 \begin{equation}\label{local-est3}
	\norm{A_q\phi}_{\HbHone^s} \leq  C \norm{\phi}_{\HbHone^{-N}}+ C
        \norm{\tP \phi}_{\HbHonestar^{s+1}} + C \norm{\crap \phi}_{\HbHone^s}.
      \end{equation}
     Here $\crap$ is microsupported in a region (away from $\pa X$, and
      close to incoming, hence not outgoing)  in which we have already
      obtained control by the estimates we have just obtained on $(\fcal_O)^c$.  Hence we may estimate
      $$
	\norm{\crap\phi}_{\HbHone^s} \leq  C \norm{\phi}_{\HbHone^{-N}}+ C
        \norm{\tP \phi}_{\HbHonestar^{s+1}},
        $$
        i.e., we may drop this term from the estimate
        \eqref{local-est3} to get \eqref{localest} for $q \in  \Tbdotstar_{\pa X} X$.

Finally, we can treat the case of $q \in \fcal_O\cap X^\circ$. We note
that in the estimate \eqref{localest} at the compressed
cotangent bundle over the boundary, the operator $A_q$ is elliptic on a
\emph{neighborhood} of $\Tbdotstar_{\pa X} X$, and this includes
points arbitrarily close to $\pa X$ along every bicharacteristic in
$\fcal_O$.  Thus by ordinary propagation of singularities, we control the
whole of $\fcal_O \cap \proj^{-1}(K')$ by the same right-hand side:
for any given point $q\in \fcal_O \cap
X^\circ$ there is $A_q$ elliptic at $q$ such that \eqref{localest} holds.
\end{proof}

In what follows we let $\norm{\bullet}_{m,l}$ denote the $H_b^{m,l}$ norm.
\begin{corollary}\label{cor:b-estimates}
For any $s \in \RR,$ $N \in \NN,$ there exists $C$ such that for all $\phi \in \dot\Hone(K')$,
\begin{equation}\label{lowerbound-b}
\norm{\phi}_{{s+1},0} \leq C \norm{\phi}_{{-N,1}}+ C
\norm{\tP^{(*)} \phi}_{s,0}.
\end{equation}
  \end{corollary}
  \begin{proof}
This follows directly from Proposition~\ref{prop:Fredholm1}, together
with the inclusions of spaces in Lemma~\ref{lemma:spaces1}.
    \end{proof}

    Equation \eqref{lowerbound-b} is not quite sufficient to use in
    the solvability argument due to the fact that $H_b^{s+1,0}$ does
    not embed compactly into $H_b^{-N,1}$. Below we perform additional
    analysis that ultimately yields an estimate similar to
    \ref{lowerbound-b} but where the function space on the left hand
    side embeds compactly to the space in which the remainder on the
    right hand side is measured.
  
In what follows, we deal with distributions in $\E'(K')$ by
considering the $t$ variables to lie in an interval: if $u
\in \E'(K')$ we may instead, by a mild abuse of notation, view it as 
$$
u \in \E'([-\pi L/2, \pi L/2]_t \times [0, \infty)_r \times S^1_\varphi),
$$
with the underlying space equipped with the same volume form as before.

\begin{lemma}\label{lemma:diophantine}
For all $s>0$ there exists $\ep>0$ such that the operator $\fib$ acting on distributions
in $\dot{H}^s([-\pi L/2, \pi L/2] \times S^1_\varphi)$ satisfies
  $$
\langle \fib^2\phi,\phi \rangle_{H^s} >\ep\norm{\phi}_{H^s}^2.
  $$
  \end{lemma}
The result is a kind of Poincar\'e inequality, and the case $s=0$ can be proved in
the usual manner by
writing
$$
\abs{\phi(t,\varphi)}^2=\int_{-\pi L/2}^{t/\C} (d/ds) \abs{\phi(\C s, \varphi-t/\C +s)}^2\, ds.
$$
and employing Cauchy--Schwarz.  An even shorter alternative, however, is
as follows.
  \begin{proof}
   For notational convenience, we shift to work on $t \in [0, \pi
   L]$ rather than $[-\pi L/2, \pi L/2]$.
    We may then use the Fourier basis $ \phi_{km}\equiv  \frac{1}{\sqrt{2\pi}} \sin ( k
t/L)e^{im\varphi}$ with $k \in \{1,2,3,\dots\}$ $m \in \ZZ.$
Then
$$
\sqrt{2\pi} \fib^2 \phi_{km} = \bigg(\frac{\C^2 k^2}{L^2} +m^2\bigg) \sin ( kt/L)
e^{im \varphi}- 2 \frac{\C k im}{L} \cos( kt/L) e^{ im \varphi}
$$
hence
$$
\ang{\fib^2 \phi_{km}, \phi_{km} } = \frac{\pi L}{2} \bigg(\frac{\C^2
  k^2}{L^2} +m^2\bigg) \geq \frac{\pi \C^2}{2 L}\norm{\phi_{km}}^2.
$$
This establishes the result for $s=0$; the more general result
then follows since $\fib^2$ commutes with $(\Id + D_t^2 + D_\varphi^2)^{s/2}$.
\end{proof}

We now prove an estimate that allows us to trade derivatives for
decay at $r=0$ in compactly supported solutions to $\tP^{(*)}u=f.$
      \begin{lemma}\label{lemma:weighted}
        Let $m \in \RR$ and let $u \in \dot{H}_b^{m+2}(K').$
Then
        $$
\norm{u}_{m,2}\lesssim \norm{P^{(*)} u}_{m,0}+ \norm{u}_{{m+2},0}
        $$
      \end{lemma}
      \begin{proof}
The estimate is trivial away from $r=0,$ hence it suffices to prove it
for $u$ replaced by $\chi u$ with $\chi$ a cutoff near $r=0$ and equal
to $1$ near $r=0.$  As the commutator term $[P^{(*)},\chi]u$ is an
element of $\Psibf^1,$ the resulting term can be absorbed in the
$H_b^{m+2,0}$ term on the RHS, 
so in fact it will suffice to
simply prove the result for $u$ supported in $r$ small.

Treating the $D_t^2$ and $\pert$ terms in $P$ as an error term,
\begin{equation}\label{indicial}
r^{-2} \fib^2 u+ r^{-2}(rD_r)^2 u=D_t^2 u+\pert u +P^{(*)} u.
\end{equation}
Since $D_t^2 \in \Psibf^2,$ $\pert \in \Psibf^1,$ we can estimate the $H_b^{m}$ norm of the RHS by $
\norm{u}_{m+2,0}+\norm{P^{(*)} u}_{m,0}.$  It thus
suffices to show that the $H_b^m$ norm of the LHS controls
the $H_b^{m,2}$ norm of $u$.

We rewrite equation \eqref{indicial} as
$$
(\fib^2 + (rD_r)^2) u=g \equiv r^2 (D_t^2 u +P^{(*)} u + \pert u) \in r^2H_b^m.
$$
We now Mellin transform this equation in $r$, using the convention
$$
\M f (\xi) \equiv \int_0^\infty f(r) r^{-i\xi-1}\, dr.
$$
This yields
$$
(\fib^2+\xi^2)\M u (\xi,t,\varphi)=\M g (\xi,t,\varphi)
$$
where we recall that we are viewing $(t,\varphi)$ as lying in a sufficiently large
set of the form $(-\pi L/2, \pi L/2) \times S^1_\varphi$.

By Lemma~\ref{lemma:diophantine}, $\fib^2>\ep$ as an operator on $\dot{H}^m$ hence we
may invert $\xi^2+\fib^2$ with uniform bounds (and holomorphy)
in $\xi$ to get
\begin{equation}\label{ug}
\norm{\M u(\xi)}_{H^m}\lesssim \norm{\M g (\xi)}_{H^m},
\end{equation}
uniformly in $\xi,$ with the norms being Sobolev norms in
$(t,\varphi).$
We now recall the characterization of b-Sobolev spaces by Mellin
transform in \cite[Section 5.6]{Me:93}.  Since $g \in r^2 H_b^m$ its
Mellin transform is holomorphic in $\Im \xi>-2$ (with the $-2$
corresponding to the $r^2$ weight, which simply shifts the imaginary
part of the Mellin transform parameter).  The squared weighted Sobolev
norm of $g$ is equivalent
to
$$
\int_{-\infty}^\infty\ang{\xi}^m \norm{\M g(-2i+\xi)}_{H^m}^2 \, d\xi.
$$
By \eqref{ug},
$$
\int_{-\infty}^\infty\ang{\xi}^m \norm{\M u(-2i+\xi)}_{H^m}^2 \,
d\xi\lesssim \int_{-\infty}^\infty\ang{\xi}^m \norm{\M g(-2i+\xi)}_{H^m}^2 \, d\xi,
$$
and this implies a corresponding inequality of b-Sobolev norms:
$$
\norm{u}_{m,2} \lesssim \norm{g}_{m, 2}\leq \norm{D_t^2+P^{(*)} u
  +\pert u}_{m,0},
$$
as desired.
        \end{proof}

    We now upgrade \eqref{lowerbound-b} at the cost of using a stronger norm
    on $P^{(*)}\phi$ on the RHS. For brevity, we define Hilbert spaces
    $$
    \hilbX=\dot{H}_b^{s,0}(K'),\quad \hilbY =
   \dot{H}_b^{s+1,0}(K')\cap \dot{H}_b^{s-2,2}(K').
    $$
    Then adding a sufficiently small multiple of the inequality in
    Lemma~\ref{lemma:weighted} with $m=s-2$ to \eqref{lowerbound-b}
    to be able to absorb the $\norm{\phi}_{s,0}$ term on the RHS of
    the former into
    the LHS of the latter yields
	\begin{equation}\label{soupeduplb2}
		\norm{\phi}_{\hilbY} \leq C \norm{\tP^{(*)} \phi}_{\hilbX}+ C
		\norm{\phi}_{-N,1}, \quad \supp \phi \subset K'.
	\end{equation}

 We crucially note the following compact embedding:
\begin{lemma}\label{lemma:compactembedding}
	$\displaystyle
	\hilbY\hookrightarrow
	\dot{H}_b^{-N,1}(K')$ is a compact embedding for $N>2-s.$
\end{lemma}
\begin{proof}
  By interpolation
  $$
          \hilbY=
\dot{H}_b^{s+1,0}(K')\cap \dot{H}_b^{s-2,2}(K')\subset
\dot{H}_b^{s-1/2-3\delta/2, 1+\delta}(K'),\quad |\delta| \le 1.
$$
Choosing $\delta \in (0, 1]$ then gives a compact embedding
$$
\dot{H}_b^{s-1/2-3\delta/2, 1+\delta}(K')\hookrightarrow \dot{H}_b^{-N, 1}(K'),
$$
since the space on the LHS has both greater differentiability when $N >s-2$  and has
greater decay
at $r=0$ (see, e.g., \cite[Lemma 6.6]{MeMe:83}).
\end{proof}
	
	Thus equation $\eqref{soupeduplb2}$ establishes an inequality with the functional analytic properties that will be necessary for the solvability argument below.
  
     Now we consider the nullspace of $P_W^*$, which will be the finite dimensional space $N$ in the statement of Theorem \ref{theorem:forward}.  Let $$N\equiv N(\tP^*)\equiv \{ u \in \dot{\Hone}(K'): \tP^* u=0\}.$$
      \begin{lemma}\label{lemma:N}
$N(\tP^*)$ is
a finite-dimensional subspace of $\HbHone^\infty.$  If Hypothesis
\ref{hyp:analytic} or \ref{hyp:tphi} holds, then the elements of $N(\tP^*)$ are
supported in $\{r>R\}.$
        \end{lemma}
        \begin{proof}
Let $u \in N(\tP^*).$ By our global propagation of singularities theorem, Proposition~\ref{prop:Fredholm1},
$$
u \in \HbHone^\infty(X).
$$
Furthermore, we employ \eqref{soupeduplb2} to see that by Lemma \ref{lemma:compactembedding}, the unit ball in $N(P_W^*)$ is compact in the $H_{b}^{-N,1}$ topology, so $N(P_W^*)$ is finite dimensional. 

We now show that under our stronger hypotheses, $N$ is in fact trivial.
Recall that $K'=\{r<2R,\ t \in [-T', T']\},$ so
that $u=0$ for $t \leq -T'.$

First, consider the case of the \ref{hyp:analytic} assumption. For any $\ep>0,$ fix a bump function $\chi_0(r)$ such that
$$
\chi_0(r)>0,\ r \in (\ep, \abs{\C}-\ep),\quad \supp \chi_0 \subset (\ep/2, \abs{\C}-\ep/2).
$$
For $s \in [0, \infty),$ let $\hypsurf_s$ denote the hypersurface
$$
\hypsurf_s = \{(t,r,\varphi): t=-T' +s \chi_0(r),\ \varphi \in S^1\}.
$$
We compute
$$
N^* \hypsurf_s =\spn\{dt-s \chi_0'(r) \, dr\}=\{\xib=r s \chi_0'(r) \,
\taub,\ \etab=0\}.
$$
Thus, on $N^*\hypsurf,$
$$
p=\sigma_b^2(P)=\taub^2- s^2
(\chi_0'(r))^2\taub^2-\frac{\C^2\taub^2}{r^2} < 0\quad \text{ for }
\taub\neq 0.
$$
Hence the whole family $\hypsurf_s$ is noncharacteristic, and by Fritz
John's global Holmgren Theorem, subject to Hypothesis~\ref{hyp:analytic}
(analyticity), we obtain $u=0$ on $\hypsurf_s$ for all $s$.  Taking $s
\to \infty$ we thus find that for any $\ep>0,$
$$
u=0 \text{ on } \{\ep<r<\abs{\C}-\ep\},
$$
hence by continuity, $u\equiv 0$ on $\{r \in [0, \abs{\C}]\}.$

Now let
$\tilde{\chi}(t)\in \mathcal{C}_c^\infty$ with
$\tilde{\chi}(t)=1$ for all $t \in [-1,1].$  Let $\Gamma$ be a smooth nondecreasing
function on $[0,\infty)$ with $$\Gamma(0)=0,\ \Gamma(s)=R-\abs{\C} \text{ for }
s \geq
1.$$
For any $\mu>0,$ set $$\chi_s(t) = \abs{\C} +\Gamma(s \tilde{\chi}(t/\mu)).$$
Note $\Gamma(s \tilde{\chi}(t/\mu))$ is compactly supported in $t,$ uniformly in $s\geq
0;$ if we take $\mu>T'$ then for each $t \in [-T', T']$,
$\chi_s(t) =R$ whenever $s\geq 1$.  Moreover, for any $s,\mu$
$$
\abs{\chi_s'(t)} \leq \sup \abs{\Gamma'} \sup \abs{\tilde\chi'}s \mu^{-1};
$$
by enlarging $\mu$ further we may thus ensure that
$$
\sup \abs{\chi_s'(t)} \leq 1,\ \text{for all } s\in [0,1].
$$

Now we again apply the global Holmgren theorem, this time for the
family of surfaces $\tilde{\hypsurf}_s= \{r=\chi_s(t)\}$, $s \in [0,1]$.  Then
$$
N^*\tilde{\hypsurf}_s = \{\taub= r^{-1} \chi_s'(t)\xi,\ \etab=0\}
$$
hence on this set
$$
p= -\frac{\xib^2}{r^2} \big(1+ (\chi_s'(t))^2  ((\C^2/r^2)-1)\big).
$$
We have $p <0$ for $\xi \neq 0$ since $r\geq \abs{\C}$ on $\hypsurf_s$
and $\abs{\chi_s'}<1$.
Hence again the surfaces are noncharacteristic and the coefficients of
$\tP^*$ are analytic on $\tilde{\hypsurf}_s$ for all $s \geq 0$ since we have
arranged that they all lie in $\{r\leq R\}.$
Thus, once again by the
global Holmgren theorem we find that $u=0$ on $\tilde{\hypsurf}_s$ for all $s,$
and these surfaces sweep out all of $K' \cap \{r \in [\abs{\C},R] \}.$  Thus we have
obtained $u=0$ on $K' \cap \{r\leq R\},$ as asserted.  (Note that
similar arguments to this one, using the Holmgren theorem, were employed in \cite{Ba:02} in the
proof of Theorem~3.5.)

We now turn instead to the case of Hypothesis~\ref{hyp:tphi} on the
perturbation $\pert.$  Since $[\pert, \pa_\varphi]=0,$ we may decompose $u \in
N(\tP^*)$ in angular modes, $$u=\sum u_k e^{ik\varphi},$$
where $u_k$ solves the mode-by-mode equation
\begin{equation}\label{mode}
\big(\frac 1{r^2} (\C\pa_t+ik)^2 - \pa_t^2 + \pa_r^2 + \frac 1r \pa_r
+\pert_k\big) u_k(t,r)=0,\ \text{ for } r\leq R
\end{equation}
with $\pert_k$ again first-order.  We have thus reduced to (a perturbation
of) the equation considered in \cite{MoWu:21}, and we proceed as in
that paper: First, note that the operator on the LHS of \eqref{mode}
is now \emph{elliptic} in $t,r$  for $t<\C,$ hence by unique continuation for
elliptic equations, $u_k=0$ for $r<\C.$  Now we Fourier transform in
$t$ and use $t$-invariance of $\pert_k$ to find
\begin{equation}\label{modefourier}
\big(-\frac 1{r^2} (\C\tau+k)^2 +\tau^2 + \pa_r^2 + \frac 1r \pa_r
+\hat{\pert}_k\big) \hat u_k(\tau, r)=0,\ \text{ for } r\leq R
\end{equation}
where $\hat{\pert}_k(\tau, r)$ is a first order operator in $r$ obtained as the Fourier conjugate of $\pert_k$ (i.e., $D_t$ turns
into $\tau$).  By Picard--Lindel\"of, since the function $\hat
u_k(\tau,r)$ vanishes for $r<\C,$ it vanishes identically on $r\leq
R.$  Thus, recovering $u$ by Fourier synthesis, we find that $u=0$ for
$r\leq R$ as well,
and we have obtained the result under Hypothesis~\ref{hyp:tphi}.
          \end{proof}

          Now we turn to the solvability argument, which is inspired
          by the work of Duistermaat--H\"ormander \cite[Theorem 6.3.1]{DuHo:72}; see also \cite{MoWu:21} for
the analogous argument in the mode-by-mode case.

  Assume  $\phi \perp N(\tP^*).$  Then we
claim that \eqref{soupeduplb2} (applied to $\tP^*$) can be replaced by
\begin{equation}\label{betterlowerbound}
  \norm{\phi}_{\hilbY} \leq C \norm{\tP^{*} \phi}_{\hilbX},\quad
  \supp \phi \subset K',
\end{equation}
i.e., the error term involving $\phi$ on the RHS can be dropped.  To
see this, note that if \eqref{betterlowerbound} fails then there
exists a sequence of $$\phi_j\in \hilbY \cap N(\tP^*)^\perp$$
with \begin{equation}\label{normalize}\norm{\phi_j}_{\hilbY} =1,\ \norm{\tP^* \phi_j}_{\hilbX} \to 0.\end{equation}
Extracting a subsequence in $\hilbY,$ converging
weakly in that space to $\phi \in  \hilbY \cap N(\tP^*)^\perp,$ we then obtain
$$
\tP^* \phi_j \to \tP^* \phi
$$
in the distributional sense.  Thus since $\tP^* \phi_j\to 0$ in $\hilbX$ we obtain $\tP^*
\phi=0,$ which implies $\phi=0.$ On the other hand, by Lemma
\ref{lemma:compactembedding},  
$\phi_j$ is strongly convergent in $\dot{H}_b^{-N,1}(K'),$ hence the limit $\phi$ must have ${H}_b^{-N,1}(K')$ norm
bounded below by $1,$ by \eqref{soupeduplb2} and \eqref{normalize}.  Thus we obtain a contradiction, and this yields
the improved estimate \eqref{betterlowerbound}.

Now for $f \in \hilbY^*\cap N(\tP^*)^\perp,$ consider the map
$$
T: \tP^* \phi \mapsto \ang{\phi, f}.
$$
By the constraint on $f,$ this map is well-defined on the range of $\tP^*$ on
$\mathcal{C}_c^\infty((K')^\circ)$, considered as a subset of
$\hilbX.$  The estimate \eqref{betterlowerbound} then implies
$$
\abs{T(\tP^* \phi)} \leq C \norm{\tP^* \phi}_{\hilbX} \norm{f}_{\hilbY^*}.
$$
By Hahn--Banach, we now extend $T$ to a map defined on all of $\hilbX$, satisfying
$$
\abs{T\psi} \leq C \norm{\psi}_{\hilbX} \norm{f}_{\hilbY^*}.
$$
Thus by the Riesz Lemma, there exists $u \in \hilbX^*$ such that
$$
T\psi=\ang{\psi, u},
$$
hence for all test functions $\phi$ we certainly have
$$
\ang{\tP^* \phi, u} = \ang{\phi, f}
$$
hence $u$ solves $\tP u=f.$
Since $\tP=P$ on $K,$ this certainly implies that $Pu=f$ on $K$.

In particular, then, given $f \in N^\perp \cap \hilbY^*$, we can solve $Pu=f$ on $K^\circ$ for $u \in \hilbX^*$.
We note that $$\hilbY^* = (\dot{H}_b^{s+1,0}(K')\cap \dot{H}_b^{s-2,2}(K'))^*
\supset H_b^{-s-1,0}((K')^\circ),$$
while
$$
\hilbX^* = H_b^{-s}((K')^\circ);
$$
setting $s=-m-1$ proves the existence of a solution in the desired
space. By Lemma~\ref{lemma:N}, if the stronger hypothesis 
Hypothesis~\ref{hyp:analytic} or \ref{hyp:tphi} holds, then
orthogonality to $N$ is no constraint on $f$ (since $f$ is by hypothesis
supported in $\{r\leq R\}$).

To prove the wavefront set relation \eqref{WFrelation} we begin by
establishing the 
stronger statement \eqref{betterWFrelation}:
$$
\WFbh u\backslash \WFbhstar f \subset \fcal_O\cup \Phi_+ (\WFbhstar f\cap \Sigma).
$$
At points not in $\Sigma\cup \Tbdotstar X$ this follows from
microlocal elliptic regularity (Proposition~\ref{prop:ellipticreg}).
At interior points in $\Sigma \backslash \fcal_O,$ it follows from interior
propagation of singularities. Indeed, by Lemma \ref{lemma:bichars}, at
every point in $\proj^{-1}(K)\setminus \mathcal{F}_O$, the
(asymptotically) backwards in time flow through that point eventually
hits the elliptic set of $W$ while remaining within $K'$, where $u$
solves $P_W u=f$.  Owing to our choice of signs for
$W$, regularity propagates (asymptotically) forward in time. Thus a
point in $\Sigma \setminus \mathcal{F}_O$ is only in the LHS if it
reaches $\WFbhstar f \cap \Sigma,$ where there may be wavefront set,
backwards in time. The weaker statement \eqref{WFrelation} then
follows from Lemma~\ref{lemma:WFinclusions}

To obtain uniqueness, we note that subtracting two solutions $u_0$,
$u_1$ yields $w=u_0-u_1$ such that
$$
\begin{aligned}
   &\tP w=0 \text{ on } K, \\ &\WFbh w \subset \fcal_O\cup \Phi_+ (\WFbhstar f\cap \Sigma).
\end{aligned}
$$
At every point $\rho\notin \fcal_O$, backward flow eventually
reaches $T^* K \backslash T^* K_0$, before leaving $K$ entirely.  (Recall that $f$ is supported in
$K_0$.)  At such points, neither $u_0$ nor $u_1$ has wavefront set,
since
$$
\WFbh u_\bullet \subset \fcal_O\cup \Phi_+ (\WFb f\cap \Sigma),
$$
and these points are neither outgoing nor in the forward flowout of $\WF f$.  Hence by
interior propagation of singularities for the equation $\tP u=0$,
$\rho \notin \WF w$, and the asserted uniqueness follows.

\appendix

\section{The uniform b-calculus}\label{appendix:bcalculus}

The usual construction of the b-pseudodifferential calculus, e.g.\ as
in \cite{Me:93}, is in the context of compact manifolds with
boundary.  Here we work on $X=[\RR^3; \strng]$, which is noncompact (recall $\strng = \{ x_1 = x_2 = 0 \}$). The noncompactness in $r$ is of no consequence here, as we always do
our estimates in a neighborhood of $\strng$ (or its lift to the
blowup), but the noncompactness in $t$ is a more serious issue and
worthy of comment.

Happily, the treatment of the subject by H\"ormander in \cite[Section 18.3]{Ho:07}
begins by developing the calculus on a half-space with just the sort
of uniform symbols estimates that we require here.  While the later
passage to manifolds in that work (Definition 18.3.18) is only phrased
in terms of local estimates (since Definition 18.2.6, giving the
conormal distributions used here, relies on \emph{local} Besov
spaces), we may still make use of the half-space construction here in order
to work on $$X\equiv [0, \infty)_r \times \RR_t \times
\RR_\varphi/2\pi \ZZ$$ by identifying functions on $X$ with functions
on $\RR^3_+$ that are $2\pi \ZZ$-periodic in the $\varphi$ variable.

Thus, following \cite{Ho:07}, we may define the Schwartz kernel of the
b-quantization of a symbol $a \in S_\unif^m$ as
\begin{equation}\label{bquantization}
\kappa(\Op_b(a)) \equiv \int e^{i [(r-r')\xi/r
  +(t-t')\tau + (\varphi-\varphi)' \eta]} a(r,t,\varphi, \xi, \tau,
  \eta) \, d\xi d\tau d\eta \abs{dt\, d\varphi \, dr/r}.
\end{equation}
 Here we have taken the form of the kernel (18.3.4) from \cite{Ho:07}
 with $x_n$ denoted $r,$ and made a change of fiber variable; we include the half-density factor necessary
 to make the operator act on functions (with the $r^{-1}$ arising
 naturally from change of fiber variables).  The $\kappa$ denotes the
 Schwartz kernel of the operator in question.

Recall that the class of symbols $a$ considered in \cite{Ho:07} and
used here \eqref{Sunif} are those
 satisfying the uniform (in $r, t,\varphi$) estimate
\begin{equation}\label{bsymbol}
\big \lvert\pa_{r,t,\varphi}^\alpha \pa_{\xi,\tau, \eta}^\beta a\big
\rvert \leq C_{\alpha,\beta, N} \ang{(\xi,\tau,\eta)}^{m-\abs{\beta}}\ang{r}^{-N}.
\end{equation}
Owing to the
periodicity in $\varphi,$ we make the additional requirement
\begin{equation}\label{periodic}
a(r,t,\varphi+2 \pi
k,\xi,\tau,\eta)=a(r,t,\varphi,\xi,\tau,\eta),\quad k \in \ZZ.
\end{equation}
And following \cite{Ho:07}, in order that this quantization produce a
sensible operator on $r\geq 0,$ we also require the ``lacunary
condition''
\begin{equation}\label{lacunary}
\fcal_{\xi \to w} a =0, \quad \text{ for } w\leq -1,\ r \geq 0.
\end{equation}
\begin{remark}
In the language of blowups, we remark that the lacunary condition means that
all derivatives of $\kappa(A)$ vanish at $s=0$ where $s=r'/r$
is a smooth variable along the interior of the front face of the
blowup $[X \times X; (\pa X)^2].$  Note that we may view the $\xi$
integration in the quantization \eqref{bquantization} as the Fourier
transform in $\xi,$ evaluated as $w=s-1.$
The set $\{w=-1\},$ a.k.a.\ $\{s=0\}$ is the ``right face''
of the blowup, at $r'=0.$  Rapid vanishing at the ``left face'' where
$r=0$ is, by contrast, automatic owing to the rapid decay of the
Fourier transform of a symbol in the base variables, since $s\to
+\infty$ as we approach this face.  Hence the operators obtained in
this way are indeed locally (in $t$)  the same as those described in \cite[Definition
4.22]{Me:93}, except that here we have built a right b-density
into the definition of the operator in order to let it act on functions.
\end{remark}

\begin{definition}\label{definition:psib}
An operator $A$ is in $\Psib^m(X)$ if it can be written as $\Op_b(a)$ for
some $a$ satisfying \eqref{bsymbol}, \eqref{periodic}, \eqref{lacunary}.
  \end{definition}

We further note, as in \cite[Section 5.3.1]{Zw:12}, that quantizing
our $\varphi$-periodic symbols as in \eqref{bquantization} and applying the
result to a $\varphi$-periodic distribution $u$ results in a $\varphi$-periodic
distribution, hence the action of $\Op_b(a)$ is well-defined on
sufficiently regular and decaying functions on $X.$

For $a$ satisfying \eqref{bsymbol} and \eqref{lacunary}, the boundedness of $\Op_b(a)$ on a
half-space is \cite[Theorem 18.3.12]{Ho:07}; the result then follows
on $X = [0,\infty) \times \RR/2 \pi \ZZ \times \RR$ by applying the
proof of Theorem 5.5 of \cite{Zw:12} in the $t$-variable to deal with
periodic functions.  This of
course yields boundedness with respect to
$$
L^2(X; dr \, dt \, d\varphi)
$$
rather than the metric volume form, however.  To obtain boundedness
with respect to the metric volume form $r dr \, dt \, d\varphi,$ we
simply note the following
\begin{lemma}\label{lemma:conjugation}
For all $m, k \in \ZZ,$
$$
A \in \Psib^m(X) \Longleftrightarrow r^{-k}A r^{k} \in \Psib^m(X).
$$
\end{lemma}
\begin{proof}
  First, let $k \in \NN.$  We note that, by integration by parts in $\xi$,
  $$
\kappa(r^{-k}A r^{k}) = \Op_b\big((1-D_\xi)^k a).
$$
The symbol $(1-D_\xi)^k a$ satisfies \eqref{bsymbol},
\eqref{periodic}, \eqref{lacunary} if $a$ does, hence $r^{-k}A r^{k}
\in \Psib(X)$ for $k \in \NN.$  The case of general $k$ now follows by
duality; recall that the calculus is closed under adjoints by
\cite[Theorem 18.3.8]{Ho:07}.
  \end{proof}
  Thus we obtain
  \begin{proposition}
    An operator $A \in \Psib^m(X)$ is bounded $\Hb^{s,l}\to
    \Hb^{s-m,l}$ for all $m, s,
    l \in \RR.$
    \end{proposition}
    \begin{proof}
Since $L^2(X)$ equipped with the metric density equals $$r^{-1} L^2(r
\, dr \, dt \, d\varphi),$$ Lemma~\ref{lemma:conjugation} implies that 
operators of order zero are bounded on $r^l L^2(X)$ for all $l \in
\ZZ.$  We can then extend to non-integer $l$ by interpolation,
establishing the result for $s=m=0.$  The more general version of the
result follows by employing elliptic operators in the b-calculus as in
the usual proof on manifolds without boundary.
      \end{proof}

We could alternatively omit the lacunary condition \eqref{lacunary} on symbols as long
as we build a cutoff function into our quantization, as in the presentation of this material in Section~\ref{sectionbPseuds}.  We must then put
the residual operators into the calculus ``by hand,'' however.  We
begin by recalling the characterization of residual operators, proved
in \cite[Theorem 18.3.6]{Ho:07}.
\begin{proposition}
 The elements of $\Psib^{-\infty}$ (a.k.a.\ \emph{residual operators})
   are those operators $R$ whose Schwartz kernels satisfy the
      following estimate in coordinates on $X^2_b$ given by
      $\rho=r+r',$ $\theta=(r-r')/(r+r')$: for all $\alpha,\beta,\gamma, N,$
      $$
\big\lvert D_{t,\varphi, t', \varphi'}^\alpha D_\rho^\beta
D_\theta^\gamma \rho \kappa(R)\big\rvert \leq C_{\alpha,\beta,\gamma,N}
(1+\abs{t-t'}+ \rho)^{-N}.
$$
  \end{proposition}
(Note that the leading factor of $\rho$ is compensating for the $\rho^{-1}$
factor in the b-half-density arising e.g.\ in \eqref{bquantization}.)

Now let $\chi(s)$ equal
$1$ for $s \in (1/2,2)$ and be supported in $(1/4,4).$
\begin{proposition} \label{prop:equivalent}
  For any $a \in S^m_\unif$ the operator
\begin{equation}\label{betterquantization}
\kappa(\tOpb(a))  \equiv \int e^{i [(r-r')\xi/r
  +(t-t')\tau + (\varphi-\varphi)' \eta]} \chi(r'/r) a(r,t,\varphi, \xi, \tau,
  \eta) \, d\xi d\tau d\eta \abs{dt\, d\varphi \, dr/r}
\end{equation}
  is in $\Psib^m(X).$  Conversely every element of $\Psib^m(X)$
  differs from an operator of this form by an element of $\Psib^{-\infty}(X).$
\end{proposition}
This result follows from \cite[Lemma 18.3.4]{Ho:07}.
  \bibliographystyle{abbrv}

\bibliography{CosmicStrings}

\end{document}